\documentclass{amsart}
\usepackage{amsmath,amssymb}
\usepackage{color,mathdots}
\usepackage{stmaryrd}
\usepackage{comment}
\usepackage{hyperref}
\usepackage{IEEEtrantools}

\newtheorem{theorem}{Theorem}[section]
\newtheorem{lemma}[theorem]{Lemma}
\newtheorem{corollary}[theorem]{Corollary}
\newtheorem{fact}[theorem]{Fact}
\newtheorem{proposition}[theorem]{Proposition}

\theoremstyle{definition}

\newtheorem{remark}[theorem]{Remark}
\newtheorem{definition}[theorem]{Definition}

\def \L {\mathcal L}
\def \d {\delta}
\def \LL {\Lambda}
\def \D {\Delta}
\def \S {\Sigma}
\def \gtcf{G-\operatorname{TCF}}
\def \DCF{\operatorname{DCF}}
\def \DCFm{\operatorname{DCF}_{0,m}}
\def \dcl{\operatorname{dcl}}
\def \acl{\operatorname{acl}}
\def \alg{\operatorname{alg}}
\def \tp{\operatorname{tp}}

\def \GDCF{G\operatorname{-DCF}_{0,m}}
\def \GDF{G\operatorname{-DF}_{0,m}}

\def \GTCF{G\operatorname{-TCF}}
\def \NN {\mathbb N}
\def \V {\mathcal V}
\def \U {\mathcal U}

\def \aut {\operatorname{Aut}}
\def \td {\operatorname{trdeg}}
\def \Frac {\operatorname{Frac}}
\def \SU{\operatorname{SU}}
\def \eq{\operatorname{eq}}
\def \theo {\operatorname{Th}}
\def \gal {\operatorname{Gal}}

\def\Ind#1#2{#1\setbox0=\hbox{$#1x$}\kern\wd0\hbox to 0pt{\hss$#1\mid$\hss}
\lower.9\ht0\hbox to 0pt{\hss$#1\smile$\hss}\kern\wd0}
\def\ind{\mathop{\mathpalette\Ind{}}}
\def\Notind#1#2{#1\setbox0=\hbox{$#1x$}\kern\wd0\hbox to 0pt{\mathchardef
\nn=12854\hss$#1\nn$\kern1.4\wd0\hss}\hbox to
0pt{\hss$#1\mid$\hss}\lower.9\ht0 \hbox to
0pt{\hss$#1\smile$\hss}\kern\wd0}

\title[Differential fields with finite group actions]{Model theory of differential fields \\ with finite group actions}

\author{Daniel Max Hoffmann$^{\dagger}$}
\thanks{$^{\dagger}$SDG}
\address{Daniel Max Hoffmann, Instytut Matematyki\\
Uniwersytet Warszawski\\
Warszawa\\
Poland}
\email{daniel.max.hoffmann@gmail.com}
\urladdr{{https://sites.google.com/site/danielmaxhoffmann/}}

\author{Omar Le\'on S\'anchez$^{\ast}$}
\address{Omar Le\'on S\'anchez, Department of Mathematics, University of Manchester, Oxford Road, Manchester, United Kingdom M13 9PL}
\email{omar.sanchez@manchester.ac.uk}

\date{\today}
\subjclass[2010]{03C60, 12L12, 12H05, 12H10}
\keywords{model theory, differential fields, difference fields, group actions}

\begin{document}

\maketitle

\begin{abstract}
Let $G$ be a finite group. We explore the model theoretic properties of the class of 
differential fields of characteristic zero in $m$ commuting derivations equipped with a $G$-action by differential field automorphisms. In the language of $G$-differential rings (i.e. the language of rings with added symbols for derivations and automorphisms), we prove that this class has a model-companion -- denoted $\GDCF$. We then deploy the model-theoretic tools developed in the first author's paper \cite{Hoff3} to show that any model of $\GDCF$ is supersimple (but unstable when $G$ is nontrivial), a PAC-differential field (and hence differentially large in the sense of the second author and Tressl \cite{OmarMarcus}), and admits elimination of imaginaries after adding a tuple of parameters. We also address model-completeness and supersimplicity of theories of bounded PAC-differential fields (extending the results of Chatzidakis-Pillay \cite{ChPi} on bounded PAC-fields).
\end{abstract}


\section{Introduction}

The model theory of group actions on fields, by field automorphisms, has been studied for at least the past three decades. For instance, for arbitrary groups by Sj\"ogren \cite{Sjo}; for $\mathbb Q$ by Medvedev \cite{med1}; for free groups $F_n$ by Chatzidakis and Hrushovski \cite{ChaHr1}, Kikyo and Pillay \cite{KiPi}, and Moosa and Scanlon \cite{MoosaScanlon}; for finite groups by the first author and Kowalski \cite{HoffKow2018}; and for finitely generated virtually free groups by Beyarslan and Kowalski \cite{BeyKow2018}. On the other hand, the case of groups acting on differential fields, by differential field automorphisms, has so far only been considered for finitely generated free groups $F_n$. The ordinary case with $G=\mathbb Z$ was investigated by Bustamante \cite{Busta}, the partial case with $G=\mathbb Z$ by the second author \cite{OmarDCFA}, and the general case $G=F_n$ by the second author and Moosa \cite{OmarRahim}.

\medskip

This paper marks the first step towards a systematic investigation of an arbitrary group $G$ acting on differential fields (by differential fields automorphisms) -- called here $G$-differential fields. This general subject has strong connections with differential Galois theory. Indeed, one is essentially studying differential fields which admit $G$ as a differential Galois group. For instance, when $G$ is say an algebraic group over the $\Delta$-constants of a ground differential field $(K,\Delta)$, then Kolchin's strongly normal extensions with Galois group $G$ are examples of $G$-differential fields. 

\medskip

Here, we focus on the case of finite groups. Let $G$ be a finite group and $m\in \NN$. One of the main purposes of this note is to show that the theory of differential fields of characteristic zero in $m$ commuting derivations equipped with a $G$-action by differential automorphisms has a model companion, which we denote by $\GDCF$. 

\medskip

We then observe several properties (of models) of this theory. More precisely, the theory $\GDCF$ is supersimple of $SU$-rank $\omega^m\cdot |G|$ with semi-quantifier elimination and elimination of imaginaries (after naming some parameters, see Section \ref{sec:EI} for details). Furthermore, if $(K,\Delta,\Sigma)\models \GDCF$ and $K^G$ and $K^\Delta$ denote the subfields of $\Sigma$-invariants and $\Delta$-constants of $K$, respectively, then

\begin{enumerate}
    \item [(i)] the differential fields $(K,\Delta)$ and $(K^G,\Delta)$ are bounded PAC-differential fields, and the latter has a supersimple theory of $SU$-rank $\omega^m$ 
    \item [(ii)] $(K^{\operatorname{alg}},\Delta)$ and $((K^G)^{\operatorname{alg}},\Delta)$ are models of $\DCFm$,
    \item [(iii)] the difference fields $(K,\Sigma)$ and $(K^{\Delta}, \Sigma)$ are models of $G\operatorname{-TCF}$ the theory of existentially closed (pure) fields with a $G$-action \cite{HoffKow2018}.
\end{enumerate}

\smallskip

Let us mention that one (quite) surprising result here is Proposition \ref{prop:equivalent}. This proposition shows that simplicity of $(K,\Delta,\Sigma)\models \GDCF$ is equivalent to simplicity of its reduct $(K,\Delta)$, and also to simplicity of the ``pure" field $K$ (even in the case when $G$ is infinite, assuming the model-companion exists).

\medskip

Our results here extend (in characteristic zero) the results in \cite{HoffKow2018} where the theory of existentially closed fields with finite group actions ($G\operatorname{-TCF}$) was explored. In a forthcoming paper we plan to investigate the case when $G$ is a finitely generated virtually free group, extending the results in \cite{BeyKow2018} (in characteristic zero) where the theory of fields with (finitely generated) virtually free group actions is explored, and also the results in \cite{OmarRahim} where the theory of (partial) differential fields with (finitely generated) free group actions was explored.

\medskip

The current paper is organized as follows. In Section \ref{preldiff} we recall the necessary differential algebraic background. In particular, we discuss in detail characteristic sets of prime differential ideals in differential polynomial rings, as they are the key ingredient to prove, in Section \ref{companion}, the existence of the model companion theory $\GDCF$. In Section \ref{theproperties}, we prove the aforementioned model-theoretic properties of this theory. Section \ref{sec:EI} is dedicated to the proof of elimination of imaginaries for $\GDCF$ with names for some parameters. In this section (see Section \ref{boundPACdiff}) we also discuss some generalities of theories of bounded PAC-differential fields. We address model-completeness and supersimplicity for such theories, which can be viewed as an extension (to the differential context) of the results on bounded PAC-fields in Section 4 of \cite{ChPi}. Section \ref{sec:final} ends the paper with several remarks about the situation when $G$ is not finite.

\medskip

\noindent {\bf Conventions.} Unless otherwise stated, rings are unital and commutative, and fields are of characteristic zero. Also, for us $\mathbb N=\{0,1,2,\dots\}$.

\section{Differential algebraic preliminaries}\label{preldiff}
In this section, we recall the necessary background on differential polynomials, differential division algorithm, and characteristic sets of prime differential ideals. We follow Kolchin's book \cite[Chapters I and IV]{kol1} where all proofs can be found. We fix $m\geqslant 0$ and $\D=\{\d_1,\dots,\d_m\}$ a distinguished set of derivations. All rings are commutative and unital. A differential ring (field) is a pair $(R,\D)$ where $R$ is a ring (field) and $\D$ is a collection of \emph{commuting} derivations on $R$. Namely, each $\d_i:R\to R$ is an additive map satisfying the Leibniz rule and $\d_i\d_j(a)=\d_j\d_i(a)$ for all $a\in R$ and $1\leqslant i,j\leqslant m$.

Let $(R,\D)$ be a differential ring. Let $x=(x_1,\dots,x_n)$ be a tuple of differential variables. The ring of differential polynomials in $x$ over $R$ is defined formally as the polynomial ring
 $$R\{x\}:=R[\d^\xi x_i: 1\leqslant i\leqslant n, \; \xi\in \mathbb N^m]$$
 where $\d^\xi x_i:=\delta_1^{\xi_1}\dots\d_m^{\xi_m}x_i$ are formal (algebraic) variables. Then $R\{x\}$ becomes a differential ring extension of $R$ by setting 
 $$\delta_j(\delta^\xi x_i)=\delta_1^{\xi_1}\cdots \d_j^{\xi_j+1}\cdots \delta_m^{\xi_m}x_i, \; \text{ for } j=1,\dots,m.$$

From now on, we work over a differential field $(K,\D)$ of characteristic zero. We fix the so-called canonical orderly ranking on the algebraic indeterminates $\d^\xi x_i$. Namely, for $1\leqslant i,j\leqslant n$ and $\xi=(\xi_1,\ldots,\xi_m),\eta=(\eta_1,\ldots,\eta_m)\in \mathbb N^m$, we set

$$\d^\xi x_i < \d^\eta x_j \;\iff \; (\sum \xi_k,i,\xi_1,\dots,\xi_m)<_{\text{lex}} (\sum \eta_k,j,\eta_1,\dots,\eta_m).$$
Let $f\in K\{x\}\setminus K$. The leader of $f$, denoted $u_f$, is the highest ranking algebraic indeterminate that appears in $f$ (according to the above ranking). The degree of $f$, denoted $d_f$, is the degree of $u_f$ in $f$. The rank of $f$, denoted $\operatorname{rk}(f)$, is the pair $(u_f,d_f)$. The set of ranks is ordered lexicographically. The separant of $f$, denoted $s_f$, is the formal partial derivative of $f$ with respect to $u_f$.
The initial of $f$, denoted $i_f$, is the leading coefficient of $f$ when viewed as a polynomial in $u_f$. Note that both $s_f$ and $i_f$ have lower rank than $f$. For a finite subset $\Lambda\subseteq K\{x\}\setminus K$, we set $H_\Lambda:=\prod_{f\in \Lambda}i_f\;s_f$. Also, for $A\subseteq K\{x\}$ we denote by $(A)$ and $[A]$ the ideal and differential ideal, respectively, generated by $A$ in $K\{x\}$. We recall that for any ideal $I$ and element $h$ of $K\{x\}$ the radical-division of $I$ by $h$ is defined as the ideal
$$I:h^{\infty}=\{f\in K\{x\}\;:\; h^r\, f\in I \text{ for some } r\geqslant 0 \}.$$
One readily checks that if $I$ is a differential ideal then $I:h^\infty$ is also a differential ideal. 

One says that $g\in K\{x\}$ is partially reduced with respect to $f\in K\{x\}\setminus K$ if no proper derivative of $u_f$ appears in $g$; if in addition the degree of $u_f$ in $g$ is strictly less than $d_f$ we say that $g$ is reduced with respect to $f$. Given a set $\Lambda\subseteq K\{x\}\setminus K$, we say $g$ is reduced with respect to $\Lambda$ if it is reduced with respect to each $f\in \Lambda$, and we say $\Lambda$ is autoreduced if for any two distinct elements of $\Lambda$ one is reduced with respect to the other. Autoreduced sets play an important role for elimination algorithms in differential polynomial rings. For instance, the \emph{differential division algorithm} states

\begin{fact}\label{div}
Let $\Lambda\subseteq K\{x\}\setminus K$ be autoreduced. For any $f\in K\{x\}$ there is $r\geqslant 0$ and $f_0\in K\{x\}$ reduced with respect to $\Lambda$ such that
$$H_\Lambda^r \; f\; \equiv \; f_0 \; \mod [\Lambda].  $$
The differential polynomial $f_0$, while not generally unique, is typically called the remainder of $f$ by $\Lambda$.
\end{fact}

Autoreduced sets are always finite and their elements have distinct leaders. We always write them in increasing order by rank; that is, if $\{f_1,\dots, f_s\}$ is autoreduced then $\operatorname{rk}(f_i)<\operatorname{rk}(f_j)$ for $i<j$. The canonical orderly ranking on autoreduced sets is defined as follows: $\{g_1,\dots,g_r\}<\{f_1,\dots,f_s\}$ means that either there is $i\leqslant r,s$ such that $\operatorname{rk}(g_j)=\operatorname{rk}(f_j)$ for $j<i$ and $\operatorname{rk}(g_i)<\operatorname{rk}(f_i)$, or $r>s$ and $\operatorname{rk}(g_j)=\operatorname{rk}(f_j)$ for $j\leqslant s$.

While it is not generally the case that prime differential ideals of $K\{x\}$ are finitely generated as differential ideals (though they are finitely generated as radical differential ideals), something close is true: they are determined by certain autoreduced subsets -- the so-called characteristic sets. More precisely, if $P\subseteq K\{x\}$ is a prime differential ideal then a \emph{characteristic set} $\Lambda$ of $P$ is a minimal autoreduced subset of $P$ with respect to the ranking defined above. These minimal sets always exist, and determine $P$ in the sense that
\begin{equation}\label{equchar}
P=[\Lambda]:H_\Lambda^\infty.
\end{equation}

It is easy to see that if $\Lambda$ is a characteristic set of a prime differential ideal and $f\in K\{x\}$, then the remainder $f_0$ of $f$ by $\Lambda$ obtained in the differential division algorithm (see Fact~\ref{div}) is in fact unique. Furthermore, 

\begin{fact}\label{reduced}
Suppose $\Lambda$ is a characteristic set of a prime differential ideal $P\subseteq K\{x\}$ and $f\in K\{x\}$. Then, $f\in P$ if and only if the remainder $f_0$ of $f$ by $\Lambda$ is zero. In particular, if $f$ is nonzero and reduced with respect to $\Lambda$, then $f\notin P$.
\end{fact}

We now wish to recall Rosenfeld's lemma which yields a (first-order) characterisation of those finite subsets of $K\{x\}$ that are characteristic sets of prime differential ideals. An autoreduced set $\Lambda=\{f_1,\dots,f_s\}$ of $K\{x\}$ is said to be \emph{coherent} if the following is satisfied: for $i\neq j$, suppose there are $\xi,\eta\in \mathbb N^m$ such that $\d^\xi u_{f_i}=\d^\eta u_{f_j}=u$ where $u$ is least such in the ranking, then
$$s_{f_j}\d^\xi f_i \; - \; s_{f_i}\d^\eta f_j \in \; (\Lambda)_u:H_\Lambda^\infty.$$
Here $(\Lambda)_u$ denotes the ideal of $K\{x\}$ generated by the $f_i$'s and their derivatives whose leaders are of rank strictly lower than $u$. 

\medskip 

Coherency helps us reduce differential consistency problems to algebraic consistency problems (in finitely many variables); that is

\begin{fact}\label{coherent}
Let $\Lambda$ be coherent and $g\in K\{x\}$ partially reduced with respect to $\Lambda$. Then, 
$$g\; \in\;  [\Lambda]:H_\Lambda^\infty \; \iff \; g\; \in \;  (\Lambda):H_\Lambda^\infty.$$
\end{fact}

With a bit a work, the division algorithm together with the above fact yield:

\begin{fact}\label{rosenfeld}{\bf Rosenfeld's Lemma.}
Let $\Lambda$ be a finite subset of $K\{x\}$. The following are equivalent
\begin{enumerate}
\item $\Lambda$ is a characteristic set of a prime differential ideal (which by \eqref{equchar} above must equal $[\Lambda]:H_\Lambda^\infty$)
\item $\Lambda$ is coherent and the ideal $(\Lambda):H_\Lambda^\infty$ is prime and does not contain nonzero elements that are reduced with respect to $\Lambda$.
\end{enumerate}
\end{fact}

We conclude this section by pointing out (sketching rather) how Rosenfeld's Lemma yields a ``first-order axiomatisation'' of characteristic sets of prime differential ideals. More precisely, 

\begin{corollary}\label{rosenchar}
Let $\LL(u,x)$ be a finite subset of $\mathbb Z[u]\{x\}$ where $u$ is a tuple of variables (thought as parameters). Then, there is a first-order formula $\phi(y)$ in the language of differential rings $\L_{\D}$ such for any differential field $(K,\D)$ of characteristic zero we have that $K\models \phi(a)$ if and only if $\LL(a,x)$ is a characteristic set of prime differential ideal of $K\{x\}$. 
\end{corollary}
A detailed proof appears in \cite[Section 4]{Tre2005}. We just wish to give a sketch; mainly as the following fact (which is needed here) will also be used in the next section.

\begin{fact}\cite[Lemma 3.2 and Theorem 4.2(iii)]{Tre2005}\label{bounds}
Let $n, d\in\mathbb N$ and let $y=(y_1,\dots,y_n)$ be a tuple of algebraic variables. Then, there is $B=B(n,d)\in \mathbb N$ such that if $h\in K[y]$ is of total degree at most $d$ and $I$ is an ideal of $K[y]$ generated polynomials of total degree at most $d$, then the radical-division ideal $I:s^\infty \subseteq K[y]$ is generated by at most $B$-many polynomials of total degree at most $B$.
\end{fact}

The proof of Corollary~\ref{rosenchar} essentially goes as follows. Given $\Lambda(u,x)$ and $(K,\D)$. For $a$ a tuple from $K$, by Rosenfeld's Lemma, $\LL:=\LL(a,x)$ is a characteristic set of a prime differential ideal of $K\{x\}$ if and only if $\LL$ is coherent and the ideal $(\Lambda):H_{\Lambda}^\infty$ is prime and does not contain nonzero elements that are reduced with respect to $\Lambda$. All of these properties can be checked in a polynomial ring in finitely many variables (as $\Lambda$ is finite and so only finitely many of the algebraic variables $\d^\xi x_i$ appear). Then, using well known bounds in the theory of polynomial rings (see \cite{vddsch}), one readily checks that
``coherency'' and ``not containing reduced elements'' can be written in a first-order fashion. For primality of $(\Lambda):H_{\Lambda}^\infty$ one uses Fact~\ref{bounds} and again the classical bounds to check for primality in polynomial rings (for the latter one needs to have control on the degree of generators, which is where Fact~\ref{bounds} is deployed).

\section{The model companion}\label{companion}
Assume $G=\{g_1,\dots,g_\ell\}$ is a finite group and let $m\in \NN$. In this section, we work out the first-order axiom scheme for the existentially closed models of the theory $\GDF$ of differential fields of characteristic zero equipped with $m$-commuting derivations and a $G$-action by differential automorphisms.

Throughout the fixed derivation symbols will be denoted by $\D=\{\delta_1,\dots, \delta_m\}$ and we let $\Sigma=\{\sigma_g:g\in G\}=\{\sigma_{g_1},\dots,\sigma_{g_\ell}\}$ denote automorphism symbols that will be associated to elements of $G$. We set
$$\L_{\Delta}:=\L_{rings}\cup \Delta\quad \text{ and }\quad
\L_{\D, \S}:=\L_{\Delta}\cup\Sigma;$$
i.e, the language of rings expanded by 
unary function symbols denoting derivations, and 
the language of rings expanded by 
unary function symbols denoting derivations and automorphisms (henceforth called the language of $G$-differential rings).

An $\L_{\D,\S}$-structure $(R,\Delta,\Sigma)$ will be called a $G$-differential ring (field) if $(R,\D)$ is a differential ring (field) and the map $G\to Aut_\D(R)$ given by $g\mapsto \sigma_g$ is a group homomorphism. 

The class of $G$-differential fields of a given characteristic is clearly first-order axiomatisable in the language $\L_{\D,\S}$. In the case of characteristic zero we denote its theory by $\GDF$. We prove that this theory has a model companion. As $\GDF$ is an inductive theory, this is equivalent to showing that the class of existentially closed models of $\GDF$ is elementary. Our approach is a combination of the (geometric) axioms appearing in \cite{HoffKow2018} and \cite{OmarDCFA}.
Let us note here that the same notation, i.e., ``$\GDF$" and ``$\GDCF$", may be used in the case of a non-finite group $G$, but, with the exception of Section \ref{sec:final}, the case of an infinite $G$ will not occur in this paper.

Let $(R,\D,\S)$ be a $G$-differential ring. Let us fix some natural number $n>0$ and let ${\bf x}_{g_i}:=(x_{g_i,1},\ldots,x_{g_i,n})$ be a tuple of \emph{differential} variables over $R$ for $i=1,\dots,\ell$. We now define a formal $G$-differential ring structure on the differential polynomial ring $R\{\bar x\}$ where $\bar x=({\bf x}_{g_1}, \dots,{\bf x}_{g_\ell})$. We do this by extending the $G$-differential structure of $(R,\Delta,\Sigma)$, using the differential structure that $R\{\bar x\}$ is already equipped with,  and the following rules for the difference structure
$$\sigma_g\big(\delta^{\xi}x_{g_i,j} \big):=\delta^{\xi}x_{g\cdot g_i.j},$$
where $g\in G$, $1\leqslant i\leqslant \ell$, $1\leq j\leqslant n$ and $\xi\in\mathbb{N}^m$.

\medskip

Assume that $I$ is an ideal of $R\{\bar{x}\}$. Recall that $I$ is $G$-invariant if for every $g\in G$
we have that $\sigma_g(I)\subseteq I$, and $I$ is differential if for every $1\leqslant i\leqslant  m$ we have that $\delta_i(I)\subseteq I$. Note that if $I$ is a $G$-invariant differential ideal in $R\{\bar{x}\}$, then there is a canonical $G$-differential structure on the quotient $R\{\bar{x}\}/I$.
If moreover $I\cap R=\{0\}$ then the extension
$R\subseteq R\{\bar{x}\}/I$ is a $G$-differential ring extension. Also, if $(R,\Delta,\Sigma)$ is a $G$-differential ring which is an integral domain, then there is a natural $G$-differential structure on the field of fractions of $R$ which extends the $G$-differential structure of $(R,\Delta, \Sigma)$.

For $a=(a_1,\ldots,a_n)\in R^n$, we set $$
\bar\sigma(a)=(\sigma_{g_1}(a_1), \dots,\sigma_{g_1}(a_n),\sigma_{g_2}(a_1),\dots,\sigma_{g_2}(a_n), \dots,\sigma_{g_\ell}(a_1),\dots, \sigma_{g_\ell}(a_n))
$$
In other words, $\bar\sigma(a)$ is just the tuple $(\sigma_{g_i}(a_j)\,:\, 1\leqslant i\leqslant \ell,1\leqslant j\leqslant n)$
ordered in the natural way. For $f\in R\{\bar{x}\}$, as usual $f(\bar\sigma(a))$ denotes the differential polynomial $f(\bar x)$ evaluated at the sequence $\bar\sigma(a)$; namely, this is given by the differential specialization $\bar x\mapsto \bar\sigma(a)$. Finally, for a subset $A$ of $R\{\bar{x}\}$, we set 
$$\V_R^\D(A)=\{b\in R^{n\cdot \ell}: f(b)=0 \text{ for all } f\in A\}.$$

%

\begin{definition}\label{def:GDCF}
A $G$-differential field $(K,\D,\S)$ of characteristic zero (i.e. a model of $\GDF$) is a model of $\GDCF$ if 
for every $n\in\mathbb{N}_{>0}$ and $\bar x=({\bf x}_{g_1}, \dots,{\bf x}_{g_\ell})$ as above,
the following is satisfied
\begin{enumerate}
\item [($\#$)] 
for any characteristic set $\LL=\{h_1,\ldots,h_k\}$ of a $G$-invariant prime differential ideal $P\trianglelefteqslant K\{\bar{x}\}$ and any nonzero $f\in K\{\bar{x}\}$ reduced with respect to $\LL$,
there is $a\in K^{n}$ such that
$$h_1(\bar\sigma(a))=\cdots=h_k(\bar\sigma(a))=0 \; \text{ and }\; (H_{\LL}\cdot f)(\bar\sigma(a))\neq 0.$$ 
\end{enumerate}
\end{definition}

\begin{remark}\label{rem:axioms}
The axiom scheme ($\#$) in Definition \ref{def:GDCF} may be expressed also as:
\begin{enumerate}
\item [($\dagger$)] 
if $P\trianglelefteqslant K\{\bar{x}\}$ is a $G$-invariant prime differential ideal, and $f\in K\{\bar{x}\}\setminus P$,
then there is $a\in K^n$ such that
$$\bar\sigma(a)\in \V_K^\D(P)\setminus V_K^\D(f).$$ 
\end{enumerate}
\end{remark}

\begin{proof}
($\dagger$) $\Rightarrow$ ($\#$):\hspace{2mm} Assume $\Lambda=\{h_1,\dots,h_k\}$, $P$, and $f$ are as in ($\#$). Then $P$ is a $G$-invariant prime differential ideal of $K\{\bar x\}$, and using Fact~\ref{reduced} it follows that $H_{\Lambda}\cdot f\notin P$. Thus, by ($\dagger$), there is $a\in K^n$ such that $\bar\sigma(a)\in \V_K^\D(P)\setminus \V_K^\D(H_\Lambda \cdot f)$. This, since $\Lambda\subset P$, yields
$$h_1(\bar\sigma(a))=\ldots=h_k(\bar\sigma(a))=0,\quad (H_{\LL}\cdot f)(\bar\sigma(a))\neq 0.$$

\medskip

($\#$) $\Rightarrow$ ($\dagger$):\hspace{2mm} Let $P$ be a $G$-invariant prime differential prime ideal of $K\{\bar{x}\}$ and $f\in K\{\bar x\}\setminus P$. Let $\Lambda=\{h_1,\ldots,h_k\}$ be a characteristic set of $P$ and let $f_0\in K\{\bar{x}\}$ be the remainder of $f$ by $\Lambda$ (in particular, $f_0$ is reduced with respect to $\Lambda$). Also, since $f\notin P$, by Fact~\ref{reduced} $f_0$ is nonzero. 

 By \eqref{equchar} in Section \ref{preldiff}, $P=[\LL]:H^{\infty}_{\LL}$. So we have that $$\V_K^\D(\LL)\setminus\V_K^\D(H_{\LL})=\V_K^\D(P)\setminus\V_K^\D(H_{\LL}).$$ Furthermore, since $H_\LL^r f\equiv f_0 \; \mod [\LL]$ for some $r\geqslant 0$, we have
$$
\V_K^\D(\LL)\setminus\V_K^\D(H_{\LL}\cdot f_0)=\V_K^\D(P)\setminus\V_K^\D(H_{\LL}\cdot f). 
$$
By ($\#$), the former contains a point of the form $\bar\sigma(a)$ for some $a\in K^n$. This point is of course also in
$\V_K^\D(P)\setminus V_K^\D(f)$, as desired. 
\end{proof}

At this point the reader might wonder why we bother ourselves with axiom scheme ($\#$) 
when the axiom scheme ($\dagger$) seems to be the most elegant and has the most geometric nature. Let us explain this.
Unlike the algebraic case (cf. \cite{HoffKow2018}), it is not yet known whether ``primality" is a definable property on the parameters for differential ideals. Thus, it is not clear whether one can express condition ($\dagger$) above as a first-order axiom. We bypass this issue by following the strategy used in \cite{OmarDCFA}, thus the introduction of the axiom scheme ($\#$) in terms of characteristic sets of prime differential ideals. 

\smallskip

Henceforth we assume $(K,\D,\Sigma)$ is a $G$-differential field of characteristic zero. We now prove that condition ($\#$) is indeed first-order. We do this in a series of lemmas. Thus, throughout the following three lemmas $n\in\mathbb{N}_{>0}$ is fixed and $\bar x=({\bf x}_{g_1}, \dots,{\bf x}_{g_\ell})$ is as above. Recall that for any $A\subseteq K\{\bar{x}\}$ we denote by $(A)$ and $[A]$ the ideal and differential ideal, respectively, generated by $A$ in $K\{\bar{x}\}$. Also, we fix the canonical orderly ranking on the indeterminates $\bar x$, as in Section~\ref{preldiff}. 

\begin{lemma}\label{equalprime}
Let $P$ and $Q$ be prime differential ideals of $K\{\bar{x}\}$ with characteristic sets $\LL$ and $\Gamma$, respectively. Then, $P=Q$ if and only if $(\Gamma):H_\Gamma^\infty$ contains $\LL$ but not $H_\LL$ and $(\LL):H_\LL^\infty$ contains $\Gamma$ but not $H_\Gamma$.
\end{lemma}
\begin{proof}
($\Rightarrow$) Assume $P=Q$. Then $\Gamma$ is a characteristic set for $P$. It follows that $|\LL|=|\Gamma|$ and we may write $\LL=\{f_1,\dots,f_s\}$ and $\Gamma=\{g_1,\dots,g_s\}$. Furthermore, $rk(f_i)=rk(g_i)$ for $i=1,\dots, s$. Thus, any $h\in K\{\bar x\}$ is reduced with respect to $\LL$ iff it is reduced with respect to $\Gamma$. It follows that $H_\LL$ is reduced with respect to $\Gamma$ and hence, by Fact
~\ref{reduced}, not in $Q=[\Gamma]:H_\Gamma^\infty$. In particular, $H_\LL\notin (\Gamma):H_\Gamma^\infty$. Now, since $f_i\in \Gamma$ and it is reduced with respect to $g_j$, for $j\neq i$, and $rk(f_i)=rk(g_i)$, the classical algebraic division algorithm performed with respect to the leader of $g_i$ yields that $f_i\in (g_i):i_{g_i}^\infty$. It follows that $\LL$ is contained in $(\Gamma):H_\Gamma^\infty$. The proof that $(\LL):H_\LL^\infty$ contains $\Gamma$ but not $H_\Gamma$ is analogous (switching the roles of $\LL$ and $\Gamma$).

($\Leftarrow$) We prove $P\subseteq Q$ (the proof of the other containment is symmetrical). Since $\LL\subset Q$ and $\Gamma\subset P$, we get as above that $\LL=\{f_1,\dots,f_s\}$ and $\Gamma=\{g_1,\dots,g_s\}$, for some $s$, and $rk(f_i)=rk(g_i)$ for $i=1,\dots, s$. It then follows that $H_\LL$ is reduced with respect to $\Gamma$. By Fact
~\ref{reduced} we obtain that $H_\Lambda\notin [\Gamma]:H_\Gamma^\infty=Q$. Thus, 
$$[\LL]:H_\LL^\infty\subseteq Q$$
but the former is $P$. 
\end{proof}

\begin{remark}\label{algbounds}
We point that that the condition $(\Gamma):H_\Gamma^\infty$ contains $\LL$ but not $H_\LL$ can be checked in a polynomial ring in \emph{finitely many} variables. Indeed, one only needs to work with the variables actually appearing in $\LL$ and $\Gamma$. This allows us to express such containments in a first-order fashion -- as degree bounds are well known, such as the ones in Fact~\ref{bounds}.
\end{remark}

\begin{lemma}\label{invariant}
Let $\LL=\{f_1,\dots,f_s\}$ be a characteristic set of a prime differential ideal $P$ of $K\{\bar x\}$ and set $\LL_g=\{\sigma_g(f_1),\dots,\sigma_g(f_s)\}$ for $g\in G$. Then, $P$ is $G$-invariant if and only if, for all $g\in G$, we have that $(\LL_g):H_{\LL_g}^\infty$ contains $\LL$ but not $H_\LL$ and $(\LL):H_\LL^\infty$ contains $\LL_g$ but not $H_{\LL_g}$. 
\end{lemma}
\begin{proof}
Let $g\in G$. Clearly $\LL_g$ is a characteristic set for the prime differential ideal $\sigma_g(P)$. By Lemma~\ref{equalprime}, equality $P=\sigma(P)$ holds if and only the conditions in the statement hold. 
\end{proof}

We can now prove:

\begin{lemma}
Let $\LL(u,\bar x)$ be a finite subset of $\mathbb Z[u]\{\bar x\}$ where $u$ is a tuple of variables (thought as parameters). Then there is a first-order formula $\phi(y)$ in the language $\L_{\D,\S}$ such for any $(K,\D,\S)\models \GDF$ we have that $K\models \phi(a)$ if and only if $\LL(a,\bar x)$ is a characteristic set of a $G$-invariant prime differential ideal of $K\{\bar x\}$. 
\end{lemma}
\begin{proof}
From Corollary~\ref{rosenchar}, there is a first-order formula $\psi(y)$ in the differential ring language $\L_{\D}$ such for any differential field $(K,\D)$ of characteristic zero we have that $K\models \phi(a)$ if and only of $\LL(a,\bar x)$ is a characteristic of a prime differential ideal of $K\{\bar x\}$. By Lemma~\ref{invariant} and Remark~\ref{algbounds}, there is a first-order formula $\rho(y)$ in the language $\L_{\D,\S}$ such for any $G$-differential field $(K,\D,\S)$ we have that $K\models \phi\land\rho\,(a)$ if and only of $\LL(a,\bar x)$ is a characteristic for a $G$-invariant prime differential ideal of $K\{\bar x\}$.
\end{proof}

This yields

\begin{corollary}
Condition ($\#$) is expressible as a first-order scheme of sentences in the language $\L_{\D,\S}$.
\end{corollary}

We now prove that the theory $\GDCF$ from Definition~\ref{def:GDCF} is indeed an axiomatisation of the existentially closed models of the theory $\GDF$.

\begin{theorem}
A model of $\GDF$ is existentially closed if and only if it is a model of $\GDCF$.
Thus the theory $\GDCF$ is the model companion of the theory $\GDF$.
\end{theorem}

\begin{proof}
Let $(K,\Delta, \Sigma)$ be an existentially closed model of $\GDF$. We will show that it satisfies the axiom scheme ($\#$).
Suppose that $P\trianglelefteqslant K\{\bar{x}\}$ is a $G$-invariant prime differential ideal,
$\LL=\{h_1,\ldots,h_k\}$ is a characteristic set of $P$, and $f\in K\{\bar{x}\}$ is nonzero and reduced with respect to $\Lambda$ (in particular, by Fact~\ref{reduced}, $f\notin P$).
Consider $L=\Frac(K\{\bar{x}\}/P)$ with its natural $G$-differential structure, which is a $G$-differential field extension of $(K,\Delta,\Sigma)$. If we set $c=(x_1+P,\ldots,x_n+P)\in L^n$, then $\bar\sigma(c)$ is a common zero of $h_1,\ldots,h_k$ in $L$ and we have that $H_\Lambda\cdot f(\bar\sigma(c))\;\neq 0$. 
Since $(K,\Delta,\Sigma)$ is an existentially closed model, 
there exists $a\in K^n$ such that $\bar\sigma(a)$ a common zero of $h_1,\ldots,h_k$ such that $H_\Lambda\cdot f(\bar\sigma(a))\; \neq 0$.

On the other hand, now assume that $(K,\Delta,\Sigma)\models\GDCF$. We will show that $(K,\Delta,\Sigma)$ is existentially closed among models of $\GDF$.
Let $(L,\Delta',\Sigma')$ be a $G$-differential field extending $(K,\Delta,\Sigma)$ and $\varphi(x)$, with $x=(x_1,\dots,x_n)$, a quantifier free $\mathcal{L}_{\Delta,\Sigma}$-formula over $K$ with a realisation $b\in L^n$. We must show that $\varphi$ has a realisation $a\in K^n$. Since $\varphi(x)$ is quantifier free, it is equivalent to a formula of the form $\psi(\bar\sigma(x))$ where $\psi(\bar x)$ is a quantifier free $\L_\Delta$-formula over $K$ with $\bar x=({\bf x}_{g_1}, \dots,{\bf x}_{g_\ell})$. So $\psi$ is a disjunction of formulas of the form
$$w_1(\bar{x}) = \cdots = w_s(\bar{x})=0 \quad
\wedge \quad  f(\bar{x})\neq 0,
$$
where $w_1,\ldots,w_s,f\in K\{\bar{x}\}$.
Therefore, we may assume that $\bar\sigma(b)$ is a common zero of $w_1, \dots,w_s$ and $f(\bar\sigma(b))\neq 0$.
Let $P:=\lbrace g(x)\in K\{\bar{x}\}\;|\; g(\bar\sigma(b))=0\rbrace$. The ideal $P$ is a $G$-invariant prime differential ideal which does not contain $f$.
By Remark \ref{rem:axioms}, ($\dagger$) implies that  there is $a\in K^n$ such that $\bar{\sigma}(a)\in \V^{\Delta}_K(P)\setminus\V^{\Delta}_K(f)$.
In particular, $\bar{\sigma}(a)$ is a common zero of $w_1,\ldots,w_s$ such that $f(\bar{\sigma}(a))\neq 0$.
It follows that $K\models \psi(\bar\sigma(a))$ and so $K\models \varphi(a)$, as desired.
\end{proof}


\section{Model theoretic properties}\label{theproperties}
In this section, we apply several results from the first author paper \cite{Hoff3} on structures with a group action and PAC substructures. Before doing so, we need to point out that the theory $\GDF$ describes substructures of a monster model of the theory $\DCFm$ that are equipped with a $G$-action by differential automorphisms, and $\GDCF$ describes those substructures that are existentially closed. We also note that the theory $\DCFm$ is quite \emph{tame} and so fits into the framework of the first author paper \cite{Hoff3}. More precisely, $\DCFm$ is a complete $\omega$-stable theory of rank $\omega^m$
that has quantifier elimination and elimination of imaginaries,
see \cite{McGrail}.

\medskip

Henceforth, we fix a monster model $(\U, \Delta)\models \DCFm$. Let $A,B,C\subseteq \U$. By $\langle A\rangle$ we denote the differential subfield of $\U$ generated by $A$. The following facts are well known (and can be found in \cite{McGrail}).

\begin{fact}\label{factsDCF} \
\begin{enumerate}
    \item $\dcl^{\DCFm}(A)=\langle A\rangle$,
    \item $\acl^{\DCFm}(A)=\langle A\rangle^{\alg},$
    \item $A\ind^{\DCFm}_C B$ (the forking independence in $\DCFm$) holds if and only if
    $\langle AC\rangle$ is free from $\langle BC\rangle$ over $\langle C\rangle$ (i.e. every finite tuple from
    $\langle AC\rangle$ which is algebraically independent over $\langle C\rangle$ remains algebraically independent over $\langle BC\rangle$).
\end{enumerate}
\end{fact}

\begin{remark}
By \cite{OmarJames}, the theory $\DCFm$ does not have the finite cover property. Furthermore, by \cite[Theorem 8.6]{HoffLee}, $\DCFm$ enjoys the boundary property B(3). We note that in \cite{HoffLee} only the case case $m=1$ is considered; however, the proof there works verbatim in the partial case.
\end{remark}

\subsection{PAC-differential fields} To obtain a more precise model-theoretic description of $\GDCF$, we will use the notion of PAC substructures. These substructures of models of a given theory can be more or less understood as being one step away from being a model. For instance, in the case of $\DCFm$ (or any $\omega$-stable theory for that matter) the algebraic closure of a PAC substructure is a model of $\DCFm$ -- see \cite{Hoff4}.

The notion of PAC substructures has defined in several places; for instance, \cite{manuscript}, \cite{PiPolk}, and \cite{Hoff3}. Here we use the following (cf. \cite{Hoff3}):

\begin{definition}
Assume that $T$ is a complete (first-order) theory and let $\mathfrak{C}$ be a monster model of $T$.
\begin{enumerate}
    \item Let $A\subseteq B$ be small subsets of $\mathfrak{C}$. We say that $A\subseteq B$ is \emph{regular}
    if
    $$\dcl^{\eq}(B)\cap\acl^{\eq}(A)=\dcl^{\eq}(A).$$
    \item Let $P$ be a small substructure of $\mathfrak{C}$. We say that $P$ is \emph{PAC} (\emph{pseudo-algebraically closed}) \emph{substructure} if for every regular extension $P\subseteq N$ ($\subseteq \mathfrak{C}$),
    the substructure $P$ is existentially closed in $N$.
\end{enumerate}
\end{definition}
\noindent
The notion of PAC substructures works well if the theory $T$ in the background is stable, has quantifier elimination and elimination of imaginaries, which is the case here as both ACF and $\DCFm$ have such properties.
If $T$ is ACF then PAC substructures are called PAC-fields, and if $T$ is $\DCFm$ then PAC substructures are called PAC-differential fields.

\medskip

In our arguments below, we will make use of properties of differentially large fields \cite{OmarMarcus}. We first recall some definitions.

\begin{definition} \
\begin{enumerate}
    \item A field $K$ is called large if $K$ is existentially closed in its Laurent series field $K((t))$. Equivalently (and in fact the usual definition \cite{Pop2014}), $K$ is large if and only if every algebraic curve over $K$ with a smooth $K$-rational point has infinitely many $K$-rational points.
    \item A differential field $(K,\D)$ is differentially large if $(K,\D)$ is existentially closed in $(K((t_1,\dots,t_m)),\D)$ where the latter derivations $\D=\{\d_1,\dots,\d_m\}$ are given by the unique (commuting) derivations extending those on $K$ that commute with infinite sums and satisfy $\delta_i(t_j)=\frac{\partial t_j}{\partial t_i}$ for $1\leq i,j\leq m$. Equivalently (and in fact the definition given \cite{OmarMarcus}), $(K,\D)$ is differentially large if and only if $K$ is large as a field and for every differential field extension $(L,\D)$, if $K$ existentially closed in $L$, then $(K,\D)$ is existentially closed in $(L,\D)$.
\end{enumerate}
\end{definition}

\begin{remark}\label{diffaut}
We recall some facts about algebraic extensions of differential fields that will be used throughout, see \cite[Chapter 2, Section 2]{kol1} for proofs. Given a differential field $(K,\D)$ (of characteristic zero) and $L$ an algebraic extension, the derivations on $K$ extend uniquely to $L$. Furthermore, if $\sigma:K\to K$ is a differential automorphism, then any field automorphism on $L$ extending $\sigma$ is differential. It follows that $\aut(L/K)=\aut_{\D}(L/K)$. In particular, this applies when $L=K^{\alg}$.
\end{remark}

The following is a compilation of results from \cite{OmarMarcus}. 

\begin{fact}\label{difflarge} \
\begin{enumerate}
    \item Any algebraic extension of a differentially large is differentially large.
    \item If $(K,\D)$ is differentially large, then $(K^{\alg},\D)\models \DCFm$. 
    \item The class of differentially large fields is elementary. Namely, there is a first-order theory in the language of differential rings whose models are precisely all differentially large fields.
    \item Let $(L,\D)$ be a differential field extension of $(K,\D)$ with both differentially large. Assume that $K$ has a model-complete theory in the language of rings expanded by constants for $K$. If $K\preccurlyeq L$, then $(K,\D)\preccurlyeq (L,\D)$.
\end{enumerate}
\end{fact}

We now note the connection between differentially large fields and PAC-differential fields. We recall that any PAC-field is large \cite{Pop2014}. 

\begin{remark}\label{PACdiff}
The following appear in \cite{OmarMarcus}.
\begin{enumerate}
\item A differential field $(K,\D)$ is PAC-differential if and only if every absolutely irreducible differential variety over $K$ has a differential $K$-rational point. In the case when $m=0$ this recovers the well known result that a field $K$ is a PAC-field if and only if every absolutely irreducible algebraic variety over $K$ has a $K$-rational point.

\item A differential field $(K,\D)$ is PAC-differential if and only if $K$ is a PAC-field and $(K,\D)$ is differentially large. 

\item The class of PAC-differential fields is elementary. Namely, there is a first-order theory in the language of differential rings whose models are precisely all PAC-differential field (in particular, being PAC-differential is first-order in the sense of Definition 2.6 from \cite{DHL}). This follows from (2). Indeed, the class of PAC-fields is known to be elementary (in the ring language) and, as we pointed out in Fact \ref{difflarge}, the class of differentially large fields is elementary (in the differential ring language).
\end{enumerate}
\end{remark}

\begin{lemma}\label{PACextalg}
Any algebraic extension of a PAC-differential field is PAC-differential.
\end{lemma}

\begin{proof}
By Lemma 4.5 from \cite{Hoff3}.
\end{proof}

\begin{lemma}\label{elementarydiff}
Let $(L,\D)/(K,\D)$ be a differential field extension of PAC-differential fields. If $K$ is bounded (namely, has finitely many extensions of degree $n$, for each $n\in \mathbb N$) and $K\preccurlyeq L$, then $(K,\D)\preccurlyeq (L,\D)$.
\end{lemma}
\begin{proof}
By Remark \ref{PACdiff}, $K$ is a PAC-field. Since $K$ is bounded, by \cite[Section 4.6]{ChPi}, the theory of $K$ in the ring language expanded by $K$ is model-complete (recall that $K$ is of characteristic zero). The result now follows by (4) of Fact \ref{difflarge}.
\end{proof}

\subsection{On the theory $\GDCF$ and its reducts}

\noindent
Let $(F,\Delta,\Sigma)$ be a monster model of $\GDCF$ embedded in the monster model $(\U,\Delta)$ of $\DCFm$. 
If $A\subseteq F$, we let
$G\cdot A:=\{\sigma_g(a)\;|\;a\in A,\,g\in G\}$.
By the results of \cite{Hoff3}, 
the following holds
\begin{enumerate}
\item $\acl^{(F,\Delta, \Sigma)}(A)=\acl^{(\U,\Delta)}(G\cdot A)\cap F=\langle G\cdot A\rangle^{\alg}\cap F$,
    \item if $(K,\Delta,\Sigma)$ and $(L,\Delta',\Sigma')$ are models of $\GDCF$ with a common substructure $E$ such that $E=\acl^{(K,\Delta,\Sigma)}(E)=\acl^{(L,\Delta',\Sigma')}(E)$, then $$(K,\Delta,\Sigma)\equiv_E (L,\Delta',\Sigma'),$$ 
    \item $\GDCF$ has semi-quantifier elimination (similarly to ACFA, see Remark 4.13 in \cite{Hoff3}); namely, every formula is equivalent to a formula of the form $\exists y\; \phi(x,y)$ where $\phi$ is quantifier-free and $\phi(a,b)$ implies $b\in \acl^{(F,\D,\S)}(a)$,
    \item $\GDCF$ codes finite sets and has geometric elimination of imaginaries,
    \item $\GDCF$ is supersimple and we have the following description of the forking independence:
    let $A,B,C$ be subsets of $F$, then
\begin{IEEEeqnarray*}{rCl}
A\ind_C^{\GDCF} B &\qquad\iff\qquad & G\cdot A\ind_{G\cdot C}^{\DCFm}G\cdot B \\
&\qquad\iff\qquad & \langle G\cdot A\cup G\cdot C\rangle\text{ is free from }
\langle G\cdot B\cup G\cdot C\rangle \\
& &\text{ over }\langle G\cdot C\rangle
\end{IEEEeqnarray*}    
where $\ind^{\GDCF}$ is the forking independence relation in $\GDCF$,
    
    \item if $(K,\Delta,\Sigma)\models\GDCF$ then $(K^G,\Delta)$ is supersimple, where $K^G:=\{a\in K\;:\;\sigma_g(a)=a\text{ for all }g\in G\}$ is the subfield of $\S$-invariants.
\end{enumerate}

\begin{remark}\label{remark:completeness}
Actually, the second item above has a slightly stronger (asymmetric) version, which is Fact 4.7 from \cite{Hoff3}.
This stronger version will be used in Lemma~\ref{lemma:reducts} below, so we state it adapted to our setting. Assume that $(K,\Delta,\Sigma)$ and  $(L,\Delta',\Sigma')$ are models of $\GDCF$ with $(K,\D)$ and $(L,\D')$ contained in the monster model $(\U,\D)\models \DCFm$. 
Let $E\subseteq K\cap L$. If $E=\langle G\cdot E\rangle^{\alg}\cap K$ 
and $\Sigma|_E=\Sigma'|_E$, then 
$(K,\Delta,\Sigma)\equiv_E (L,\Delta',\Sigma')$.
\end{remark}

We now look at the different reducts of a model of $\GDCF$ when looking at the various (sub-)fields of constants. Recall that if $(K,\Delta)$ is a differential field and $\Pi\subseteq \Delta$, then $K^\Pi$ denotes the subfield of $\Pi$-constants of $K$ (i.e., $K^\Pi=\ker{\Pi}$). Since the derivations commute, the derivations $\mathcal D:=\Delta\setminus \Pi$ on $K$ restrict to derivations on $K^\Pi$, and so $(K^\Pi,\mathcal D, \S)$ is a $G$-differential subfield of $(K,\mathcal D, \S)$.

\begin{lemma}\label{lemma:reducts}
Let $(K,\Delta,\Sigma)\models \GDCF$ and $\Pi\subseteq \Delta$. Set $\mathcal D=\Delta\setminus\Pi$ and $r=|\mathcal D|$. Then, $(K,\mathcal D,\Sigma)$ and  $(K^\Pi, \mathcal D,\Sigma)$ are models of $G$-DCF$_{0,r}$.
\end{lemma}
\begin{proof}
To prove that $(K,\mathcal D, G)\models \DCF_{0,r}$ we use the characterisation of $G$-DCF$_{0,r}$ given in Remark~\ref{rem:axioms}. Let $P$ be a $G$-invariant prime differential ideal of $K\{\bar x\}_{\mathcal D}$. Here, to avoid confusion, we use the subscript $\mathcal D$ to denote the ring of $\mathcal D$-differential polynomials. Also, let $f\in K\{\bar x\}_{\mathcal D}\setminus P$. Now let $Q$ be the radical differential ideal of $K\{\bar x\}_\Delta$ generated by $P$. Since $\Sigma$ consists of $\Delta$-automorphisms, one readily checks that $Q$ is $G$-invariant. Furthermore, by Proposition 8 of \cite[Chapter 0, Section 6]{Kolchin2}, $Q$ is prime and 
$$P=Q\cap K\{\bar x\}_{\mathcal D}.$$
It follows that $f$ is not in  $Q$. Since $(K,\Delta,\Sigma)\models \GDCF$, there is a tuple $a$ from $K$ such that 
$$\bar\sigma(a)\in \V_K^{\Delta}(Q)\setminus \V_K^\Delta(f).$$
Since $P\subseteq Q$, it readily follows that
$$\bar\sigma(a)\in \V_K^{\mathcal D}(P)\setminus \V_K^{\mathcal D}(f).$$
Thus, $(K,\mathcal D,\Sigma)\models G$-DCF$_{0,r}$.

\medskip

To prove that $(K^\Pi, \mathcal D,\Sigma)$ models $G$-DCF$_{0,r}$ we show that it is an existentially closed model of $G$-DF$_{0,r}$. Let $\phi(x)$ be a quantifier free $\mathcal L_{\mathcal D,\Sigma}$-formula over $K^{\Pi}$ with a realisation in some $(L,\mathcal D', \Sigma')$ extending $(K^\Pi,\mathcal D, \Sigma)$. Let $\Delta'=\mathcal D'\cup \Pi'$ where $\Pi'=\{\pi'_1,\dots,\pi'_{m-r}\}$ and each $\pi'_i$ is the trivial derivation on $L$. Then $(L,\Delta',\Sigma')$ is a model of $G$-DF$_{0,m}$. Now let $(F,\Delta', \Sigma')$ be a model of $\GDCF$ extending $(L,\Delta',\Sigma')$. We note that $(F,\Delta', \Sigma')$ has a realisation of $\phi(x)$ that vanishes in all the elements of $\Pi'$. Also, note that $K^{\Pi}$ is a common substructure of $(K,\Delta,\Sigma)$ and $(F,\Delta',\Sigma')$. Moreover, $K^\Pi=\acl^{(K,\Delta,\Sigma)}(K^\Pi)$ since $K^{\Pi}$ is relatively algebraically closed in $K$ (this property holds for any subfield of constants). By Remark \ref{remark:completeness}, it follows that $$(K,\Delta,\Sigma)\equiv_{K^\Pi}(F,\Delta',\Sigma'),$$ 
and hence, since $F$ has a realisation of $\phi(x)$ that vanishes in all the elements of $\Pi'$, there is a realisation of $\phi(x)$ in $K$ which vanishes in all of the elements of $\Pi$. Thus, this latter is a realisation of $\phi(x)$ in $K^{\Pi}$, as desired. 
\end{proof}

We recall that the model-companion of the theory of $G$-fields (in arbitrary characteristic) was explored in \cite{HoffKow2018} and denoted by $\GTCF$. Note that in our terminology, if we set $m=0$, then the models of $G\operatorname{-DCF_{0,0}}$ are precisely the models of $\GTCF$ of characteristic zero. Thus, setting $\Pi=\D$, the lemma above yields.

\begin{corollary}
Let $(K,\D,\S)$ be a model of $\GDCF$. Then $(K,\S)$ and 
$(K^\D,\S)$ are models of $G\operatorname{-TCF}$. 
\end{corollary}

This corollary allows us to use the results from \cite{HoffKow2018} to conclude several field-theoretic properties about models of $\GDCF$ and their subfields of invariants. We collect some of these in the following corollary; but first some definitions.

\begin{definition}
\begin{enumerate}
    \item A $G$-field $(K,\Sigma)$ is strict if $[K:K^G]=|G|$. One easily checks that being strict is equivalent to $\Sigma=\aut(K/K^G)$. And it is also equivalent to the action of $G$ on $K$ being faithful.
\item A $G$-field $(K,\S)$ is algebraically $G$-closed if the difference $G$-structure on $K$ cannot be extended to any proper algebraic extension of $K$.
\item Given a field, the absolute Galois group, denoted $\gal(K)$, is the group $\aut(K^{\alg}/K)$.
\item A field is bounded if it has finitely many extensions of degree $n$, for each $n\in\mathbb N$. Equivalently, $\gal(K)$ is small (as profinite group).
\end{enumerate}
\end{definition}

\begin{corollary}\label{fromTCF}
Let $(K,\D,\S)\models \GDCF$. Then,
\begin{enumerate}
    \item $K$ and $K^G$ are bounded PAC-fields but not pseudo-finite.
    \item $(K,\S)$ is strict and algebraically $G$-closed,
    \item The natural restriction map
    $$\gal(K^G)\to G$$
    is the universal Frattini cover of $G$.
    \item If $G$ is nontrivial, then $K$ is not algebraically closed. The measure of how far $K$ is from $K^{\alg}$ is given given by the kernel of the universal Frattini cover in (3). Indeed, 
    $$\gal(K)\cong \ker (\gal(K^G)\to G).$$
    \item If $G$ is trivial, then ${\mathcal Mod}(\GDCF)={\mathcal Mod}(\DCFm)$. Indeed, when $G$ is trivial, referring to our axioms of $\GDCF$ above, any prime differential ideal will be $G$-invariant, and hence, in this case, our axioms of $\GDCF$ yield the natural axioms of $\DCFm$ (cf. \cite{McGrail} or \cite{Tre2005}).
\end{enumerate}
\end{corollary}

The following extends to the differential context the notion of $K$-strongly PAC-field from \cite{HoffKow2018}.

\begin{definition}
Let $(K,\D)/(F,\D)$ be a differential field extension. We say that $(F,\D)$ is $K$-strongly PAC-differential if every differential variety over $F$ that is $K$-irreducible has a differential $F$-rational point. Of course, this notion implies PAC-differential.
\end{definition}

\begin{proposition}\label{strongPACdiff}
If $(K,\D,\S)$ is a model of $\GDCF$, then $(K^G,\D)$ is $K$-strongly PAC-differential. In particular, 
\begin{enumerate}
    \item [(i)] $(K,\D)$ and $(K^G,\D)$ are both PAC-differential (and hence differentially large by (2) of Remark \ref{PACdiff}).
    \item [(ii)] $(K^{\alg}, \D)$ and $((K^G)^{\alg},\D)$ are models of $\DCFm$.
\end{enumerate}
\end{proposition}
\begin{proof}
Let $V$ be a differential variety over $K^G$ which is $K$-irreducible. We may assume that $V$ is affine contained in $\U^n$. Then, the differential vanishing ideal $P=\mathcal I_\D(V/K)$ of $V$ in $K\{x\}$, with $x=(x_1,\dots,x_n)$, is prime. Now, as in Section \ref{companion}, we let $\bar x=({\bf x}_{g_1}, \dots,{\bf x}_{g_\ell})$ and regard $x$ as ${\bf x}_{g_1}$. Then, the  differential ideal $Q$ of $K\{\bar x\}$ differentially-generated by $P$ and ${\bf x}_{g_1}- {\bf x}_{g_i}$, for $i=2,\dots,\ell$, is a $G$-invariant prime differential ideal. by the axioms of $\GDCF$, we there is $a\in K$ such that
$$\bar\sigma(a)\in \V_K^\D(Q).$$
By the equations defining $Q$ we have that $a\in V$ and $a$ is a tuple from $K^G$, as desired. 

It now also follows that $(K^G,\D)$ is PAC-differential. Since $K$ is an algebraic extension of $K^G$, Lemma \ref{PACextalg} implies that $(K,\D)$ is also PAC-differential. This show (i) of the 'in particular' clause. For (ii) one can simply invoke (2) of Fact~\ref{difflarge}, since PAC-differential fields are differentially large (see (2) of Remark~\ref{PACdiff}).
\end{proof}

We now aim to characterise those $G$-differential fields that are models of the theory $\GDCF$ in terms of properties satisfied by the (differential) subfield of invariants. The following is one of the key ingredients to achieve this. 

\begin{proposition}\label{elext}
Suppose $(K,\D,\S)$ and $(L,\D,\S)$ are $G$-differential fields with $(K,\D,\S)\subseteq (L,\D,\S)$. Assume that $(K,\S)$ and $(L,\S)$ are strict and algebraically $G$-closed. If $(K^G,\D)$ and $(L^G,\D)$ are PAC-differential, then $(K^G,\D)\preccurlyeq(L^G,\D)$.
\end{proposition}
\begin{proof}
By Remark \ref{PACdiff}, $K^G$ and $L^G$ are PAC-fields. By Theorems 3.22 and 3.25 of \cite{HoffKow2018} (which use the assumptions that $(K,\S)$ and $(L,\S)$ are strict and algebraically $G$-closed), we have that $K^G$ is bounded and $K^G\preccurlyeq L^G$. Since $(K^G,\D)$ and $(L^G,\D)$ are both PAC-differential, Lemma \ref{elementarydiff} implies that $(K^G,\D)\preccurlyeq (L^G,\D)$.
\end{proof}

Here is the desired characterisation (which extends Theorem 3.29 from \cite{HoffKow2018}). 

\begin{theorem}
Let $(K,\D, \S)$ be a $G$-differential field. Assume $(K,\S)$ is strict. Then, the following are equivalent. 
\begin{enumerate}
    \item $(K,\D,\S)\models \GDCF$.
    \item $(K^G,\D)$ is $K$-strongly PAC-differential.
    \item $(K^G,\D)$ is PAC-differential and $(K,\S)$ is algebraically $G$-closed.
\end{enumerate}
\end{theorem}
\begin{proof}
We have already seen than (1) implies (2).

We prove (2) implies (3). Assume towards a contradiction that $(K,\S)$ is not algebraically $G$-closed. Then, there is an proper algebraic extension $L$ of $K$ such that the $G$-difference structure on $K$ extends to $L$. We note that any extension of $\Sigma$ to $L$ continues to commute with $\D$ (i.e., the extensions to $L$ are differential automorphisms). We also note that $L^G$ is a proper algebraic extension of $K^G$. Let $a\in L^G\setminus K^G$ and $f$ be the minimal polynomial of $a$ over $K$. Then, for any $g\in G$, we have
$$0=\sigma_g(f(a))=f^{\sigma_g}(\sigma_g(a))=f^{\sigma_g}(a),$$
where $f^{\sigma_g}$ denotes the polynomial obtained by applyng $\sigma_g$ to the coefficients of $f$. It follows that $f^{\sigma_g}=f$ and so $f$ is over $K^G$. However, $f$ has no root in $K^G$, contradicting that $(K^G,\D)$ is $K$-strongly PAC-differential.

We prove (3) implies (1). Let $(L,\D,\S)$ be model of $\GDCF$ which extends $(K,\D,\S)$. By the results above, we know that $(L,\S)$ is strict and algebraically $G$-closed. Furthermore, $(L^G,\D)$ is PAC-differential. By Proposition \ref{elext}, we have that
\begin{equation}\label{eqelementary}
(K^G,\D)\preccurlyeq(L^G,\D).
\end{equation}
Now let $v_1,\dots,v_\ell$ be a linear basis of $K$ over $K^G$ (by strictness $\ell=|G|$). Since $(K,\S)$ is strict, it follows that $v_1,\dots,v_\ell$ is also a linear basis for $L$ over $L^G$ (indeed, $K$ is linearly disjoint from $L^G$ over $K^G$). Using this basis one can induce a $G$-field structure (as in Remark 2.3 of \cite{HoffKow2018}); namely, a multiplication and a $G$-action on the differential $K^G$-vector space $((K^G)^{\ell},\D)$, where the elements of $\D$ act coordinate-wise on $(K^G)^\ell$ and so $G$ acts by differential automorphisms. Note that then $((K^G)^\ell, \D, \S)$ is a $G$-differential field isomorphic to $(K,\D,\S)$, and the structure in the former is \emph{definable} in $(K^G,\D)$. Moreover, since the $v_i$'s are in $K$, the same formulas defining multiplication and $G$-action in $((K^G)^\ell, \D, \S)$ also define a $G$-differential field structure on $((L^G)^\ell,\D)$. Again, $((L^G)\ell,\D,\S)$ is isomorphic to $(L,\D,\S)$, and is \emph{definable} in $(L^G,\D)$. By \eqref{eqelementary}, we have
$$((K^G)^\ell,\D, \S)\preccurlyeq((L^G)^\ell,\D, \S).$$
By the isomorphisms mentioned above, this yields
$$(K,\D,\S)\preccurlyeq(L,\D,\S),$$
 and hence $(K,\D,\S)\models \GDCF$, as desired. 
\end{proof}


\subsection{Supersimplicity and rank computations} As we pointed out in Corollary \ref{fromTCF}, when the group $G$ is trivial the theory $\GDCF$ coincides with $\DCFm$ and so it is $\omega$-stable of rank $\omega^m$. When the group is nontrivial, the story is slightly different. 

\begin{proposition}\label{simplenotstable}
Let $(K,\Delta,\Sigma)\models\GDCF$. If $G$ is non-trivial, then $(K,\Delta,\Sigma)$ and $(K^G,\D)$ are supersimple, but unstable.
\end{proposition}
\begin{proof}
Supersimplicity, as we pointed out above, comes from \cite{Hoff3}.
To show unstability, it is enough to conclude that the reducts $K$ and $K^G$ are unstable fields.
To see that, we use the fact that $K$ and $K^G$ are PAC-fields (by Corollary \ref{fromTCF}),
by Fact 2.6.7 from \cite{Kimsim} we know that a PAC-field of characteristic zero is stable if and only if it is algebraically closed. But in Corollary \ref{fromTCF} we pointed out that, when $G$ is nontrivial, $K$ cannot be algebraically closed (and thus neither is $K^G$).
\end{proof}

\begin{remark}\label{rem:ranks}
Before we prove the next result, we need to recall what is the $\SU$-rank of $\DCF_{0,m}$. In Lemma 5.4.13 of \cite{OmarPhD} it is shown that if $p\in S_1(\mathbb Q)$ is the generic type in $\DCFm$ then $\SU(p)=\omega^m$. By quantifier elimination, $p=tp^{\DCFm}(a)$ if and only $a$ is differentially transcendental over $\mathbb Q$. In fact, from the proof in \cite{OmarPhD}, one can deduce that $\SU(a/\mathbb Q)=\omega^m$ if and only $a$ is differentially transcedental over $\mathbb Q$. Furthermore, using Lascar inequalities, we get that for an $n$-tuple $\bar a$ we have that $\SU(\bar a/\mathbb Q)=\omega^m\cdot n$ if and only if $\bar a$ is differentially transcendental (or differentially independent rather) over $\mathbb Q$.
\end{remark}

\noindent Recall that in the case $m=0$ (i.e. no derivations) the theory $G\operatorname{-DCF_{0,0}}$ coincides with the theory $\GTCF$. In this case, the $\SU$-ranks of 
$K^G$ and $(K,\Sigma)$ were computed in Propositions 4.8 and 4.9, respectively, of \cite{HoffKow2018}. Here, when $m$ is arbitrary, we have the following counterpart:

\begin{theorem}\label{SUrankinv}
Let $(K,\Delta,\Sigma)$ be a monster model of $\GDCF$.
The $\SU$-rank of $(K^G,\Delta)$ is $\omega^m$. 
\end{theorem}

\begin{proof}
To compute the $\SU$-rank of $(K^G,\Delta)$ we need a practical description of forking independence in $(K^G,\Delta)$. Fix some small $(E,\Delta)\preceq(K^G,\Delta)$. As we have pointed out above, $\DCFm$ enjoys NFCP, the boundary property B(3), and being PAC-differential is a first-order property in $\DCFm$; hence,
we can use the forking independence description given in Remark 6.20 in \cite{HoffLee},
provided that 
$(K^G,\Delta)$ has small absolute Galois group. But this follows from the fact that $\acl^{\DCFm}(K^G)=(K^G)^{\alg}$, $\gal(K^G)=\gal_{\D}(K^G)$ (see Remark~\ref{diffaut}), and $K^G$ is bounded. 


Thus, by Remark 6.20 from \cite{HoffLee}, we get that the forking independence $\ind$ in $(K^G,\Delta)$
extended by parameters from $E$ is given by the forking independence in $\DCFm$; namely:
$$a\ind^{(K^G,\Delta)}_A B\qquad\iff\qquad a\ind^{\DCFm}_{A} B,$$
where the tuple $a$, and the sets $E\subseteq A,B$ are from $K^G$.
Using the description of the forking independence relation in $\DCFm$, 
we obtain that
$$a\ind^{(K^G,\Delta)}_{E}B\qquad\iff \langle a, E\rangle\text{ is free from }\langle EB\rangle\text{ over }E,$$
where $a\in K^G$ and $B\subseteq K^G$. 
Because the SU-rank of $\DCFm$ is $\omega^m$, we obtain that the SU-rank of $\tp^{(K^G,\Delta)}(a/E)$ is at most $\omega^m$ for every $a\in K^G$. 

\medskip

On the other hand, by saturation of $(K,\D,\S)$ and the axioms of $\GDCF$, we can find $a\in K^G$ such that $a$ is differentially transcendental over $E$. Then the $\SU$-rank of $tp^{\DCFm}(a/E)$ is $\omega^m$. It follows that $\SU$-rank of $tp^{(K^G,\D)}(a/E)=\omega^m$.  
\end{proof}

\begin{theorem}\label{theorem:SU.rank.K}
Let $(K,\Delta,\Sigma)$ be a monster model of $\GDCF$. Then, $\SU$-rank of $(K,\Delta,\Sigma)$ is equal to 
$\omega^m\cdot |G|$ (in Cantor normal form). 
\end{theorem}

\begin{proof}
As we have pointed out above, it follows from the results in \cite{Hoff3} that the following describes forking independence in $\GDCF$. Let $A,B,C$ be small subsets of $K$, then
\begin{IEEEeqnarray*}{rCl}
A\ind_C^{\GDCF} B &\qquad\iff\qquad & G\cdot A\ind_{G\cdot C}^{\DCFm}G\cdot B \\
&\qquad\iff\qquad & \langle G\cdot A\cup G\cdot C\rangle\text{ is free from }
\langle G\cdot B\cup G\cdot C\rangle \\
& &\text{ over }\langle G\cdot C\rangle
\end{IEEEeqnarray*}    
where $\ind^{\GDCF}$ is the forking independence relation in $\GDCF$.

By the axioms of $\GDCF$ we can find $a\in K$ such that the tuple
$$\bar \sigma(a)=(\sigma_{g_1}(a),\dots, \sigma_\ell(a))\in K^\ell$$
is differentially transcendental (over $\mathbb Q$ say). But then the $\SU$-rank of $tp^{\DCFm}(\bar\sigma(a))$ is $\omega\cdot|G|$. see Remark~\ref{rem:ranks}. If follows that the $\SU$-rank of $tp^{\GDCF}(a)$ equals $\omega^m\cdot |G|$.

On the other hand, for any $a\in K$, the $\SU$-rank of $tp^{\GDCF}(a)$ is at most the $\SU$-rank of $tp^{\DCFm}(\bar\sigma(a))$. But the latter (which is computed in $\DCFm$) is at most $\omega^m\cdot |G|$, again see Remark~\ref{rem:ranks}. The result follows.
\end{proof}

\begin{remark}
As mentioned earlier the theory $\GDCF$ has ``semi" elimination of quantifiers (similarly as ACFA).
One could ask whether the theory $\GDCF$ has (full) elimination of quantifiers. If $G$ is a trivial group, then, as $\GDCF$
coincides with the theory $\DCFm$, it does have elimination of quantifiers. However, if $G$ is a nontrivial (finite) group,
then elimination of quantifiers in $\GDCF$ would lead to stability, which is not possible by Proposition~\ref{simplenotstable} (cf. similar result for the theory $\GTCF$ in Remark 2.11 in \cite{HoffKow2018}). 
\end{remark}


\section{Elimination of imaginaries and bounded PAC-differential fields}\label{sec:EI}
As we already know, every completion of $\GDCF$ has geometric elimination of imaginaries (Theorem 4.36 in \cite{Hoff3}).
In this section we improve this result by pushing the proof of Theorem 4.36 from \cite{Hoff3} further as it has been done for other well known theories -- such as ACFA in \cite{acfa1}, PAC-fields in \cite{manuscript} and \cite{ChPi}, and fields with free operators in \cite{MoosaScanlon}.

\medskip

Before doing this we need to state an improved version of the Independence Theorem (see Theorem 4.21 \cite{Hoff3}); namely, the Independence Theorem over \emph{algebraically closed substructures}. To get this improved version, we will slightly modify the language by adding suitable parameters (as is done for bounded PAC-fields \cite[Section 4]{ChPi}). The reason for this extra effort is that model-theoretic algebraic closure in $\GDCF$ is not necessarily the whole field-theoretic algebraic closure; but the additional parameters give us a strong relation between these two notions of algebraic closure (see Lemma \ref{lemma:good.acl}). 

\medskip

Another key ingredient here, which is important for elimination of imaginaries, is the boundary property $B(3)$ (which as we have pointed out holds for $\DCFm$). Let us note that the absence of this property is behind the fact that the theory of Compact Complex Manifolds with an Automorphism (CCMA, \cite{CCMA}) does \emph{not} have full elimination of imaginaries (note that CCMA is also a model-complete theory with a group action and hence does have geometric elimination of imaginaries).

\medskip

We note that the strategy of our proof here is modelled by the (now standard) arguments in Section 4 of \cite{ChPi} where elimination of imaginaries is established for theories of bounded PAC-fields (after adding a suitable tuple of parameters).

\subsection{On bounded PAC-differential fields}\label{boundPACdiff} In this section we briefly point out how to extend the results from Section 4 of \cite{ChPi} to the differential context (in zero characteristic). In fact, one essentially only needs to extend the field-theoretic ingredients used in the proofs of that section; namely, Theorem 4.3 in \cite{ChPi}. In the differential language this translates to the theorem below. We first recall, from Remark~\ref{diffaut}, that for a differential field $(K,\D)$ we have $\acl^{\DCFm}(K)=K^{\alg}$ and
$$\aut_{\D}(K^{\alg}/K)=\aut(K^{\alg}/K).$$
Thus, the absolute Galois group in the sense of $\DCFm$ (i.e., $\aut_{\D}(\acl^{\DCFm}(K)/K)$) coincides with the field-theoretic absolute Galois group. Henceforth, we denote this group by $\gal(K)$.

\begin{theorem}\label{propertiesPAC}
Let $(K,\D)$ be a differential field, and let $(E,\D_1)$ and $(L,\D_2)$ be differential field extensions. Then,
\begin{enumerate}
    \item Assume $(E,\D_1)$ and $(L,\D_2)$ are PAC-differential. If there is a (topological) isomorphism $\gal(E)\to \gal(L)$ which forms a commutative diagram with the restriction maps $\gal(E)\to \gal(K)$ and $\gal(L) \to \gal (K)$, then $$(E,\D_1)\equiv_{K} (L,\D_2).$$
    \item If $(K,\D)$ and $(E,\D_1)$ are PAC-differential and the restriction map $\gal(E)\to \gal (K)$ is an isomorphism, then $(K,\D)\preccurlyeq(E,\D_1)$.
    \item If $(K,\D)$ is PAC-differential, then $\gal(K)$ is a projective profinite group.
    \item Any algebraic extension of a PAC-differential field is PAC-differential
    \item Any differential field has a differential field extension which is PAC-differential and a regular extension (in the field-sense).
\end{enumerate}
\end{theorem}
\begin{proof}
(1) follows immediately from Theorem 5.11 of \cite{DHL} (by quantifier elimination of $\DCFm$, being a \emph{sorted} isomorphism is equivalent to being an isomorphism).

(2) is a consequence of (1).

(3) By Theorem 4.4 in \cite{Hoff4}.

(4) This is Lemma \ref{PACextalg} above. 

(5) This follows from Proposition 3.6 \cite{Hoff3}.
\end{proof}

Now the results in Section 4 of \cite{ChPi} can be adapted (almost verbatim) using this differential version. To state these results we let $T_\D$ be the complete theory (in the language of differential rings $\mathcal L_\D$) of a \emph{bounded} PAC-differential field $(F,\D)$. For each $n> 1$, let $N(n)$ be the degree over $F$ of the Galois extension composite of all Galois extensions of
$F$ of degree $n$. Consider the language $\mathcal L_{\D,c}$
obtained by adjoining to $\mathcal L_{\D}$ the new constant symbols $c_{n,i}$, for $n > 1$ and  $0\leq  i < N(n)$. We let $T_{\D,c}$ be the $\mathcal L_{\D,c}$-theory obtained by adding to $T_\D$ axioms expressing: for each $n > 1$, the polynomia1 
$$x^{N(n)} + c_{n,N(n)-1}x^{N(n)-1}+\cdots+ c_{n,0}$$ 
is irreducible and the extension obtained by adding to $F$ a root of this polynomial is Galois and contains all Galois extensions of $F$ of degree $n$.

\begin{theorem}
Let $T_{\D,c}$ be as above, $(F,\D,c)\models T_{\D,c}$, and $(E,\D,c)\subseteq (F,\D,c)$. Then, 
\begin{enumerate}
    \item $\acl(E)$ equals the relative field-theoretic algebraic closure of $E$ in $F$.
    \item If $E=\acl(E)$, then $T\cup qfDiag(E)$ is complete.
    \item $T_{\D,c}$ is model-complete.
    \item (Independence theorem over $\acl$-closed sets) Let $a, b, c_1, c_2$ be tup1es from $F$ independent over E. If $tp(c_1/\acl(E)) = tp(c_2/\acl(E))$, then there is a tuple $c'$ independent from $(a,b)$ over $E$ realising $tp(c_1/\acl(Ea)) \cup tp(c_2/\acl(Eb))$.
    \item The theory $T_{\D,c}$ is supersimple, of $\SU$-rank $\omega^m$, and forking independence is given by the forking independence in $\DCFm$; namely:
$$A\ind^{T_{\D,c}}_E B\qquad\iff\qquad A\ind^{\DCFm}_{E} B,$$
with $A,B,E$ subsets of $F$.
    \item The theory $T_{\D,c}$ eliminates imaginaries.
\end{enumerate}
\end{theorem}
\begin{proof}
As we have already mentioned the proofs are parallel to those in Section 4 of \cite{ChPi}, and these are nowadays standard model-theoretic arguments. We only make brief comments on some parts of the adaptations and leave the missing details to the reader. We note that some of these (missing) arguments are spelled out in Section \ref{imagine} below (for the case of models of $\GDCF$). 

\medskip

(1) The key ingredient in this proof (see Proposition 4.5 in \cite{ChPi}) is to show that if $(K,\D)$ is PAC-differential and $(L,\D)$ is a differential field extension with $L/K$ regular, then there is a differential field extension $(K',\D)$ of $(L,\D)$ which is PAC-differential and the restriction map $\gal(K')\to \gal (K)$ is an isomorphism. To see this, one lets $(L',\D)$ be a PAC-differential field extension of $(L,\D)$ with $L'/L$ regular (given by Theorem~\ref{propertiesPAC}). It follows that $L'/K$ is also regular, and so the restriction $\gal(L')\to \gal(K)$ is surjective. Since $\gal(K)$ is projective (by Theorem \ref{propertiesPAC}), there is a closed subgroup $G_1$ of $\gal (L')$ such that the restriction induces an isomorphism $G_1\to \gal(K)$. Let $K'$ be the subfield of $(L')^{\alg}$ fixed by $G_1$. Then, since $K'$ is an algebraic extension of $L'$, $(K',\D)$ is PAC-differential. Since $\gal(K')=G_1$, the restriction $\gal(K')\to \gal(K)$ is an isomorphism. 

\medskip

(2) Here the key ingredient is to show that if $(L,\D,c)$ is a model of $T_{\D,c}$ containing $E$ then the restriction $\gal(L)\to \gal (E)$ is an isomorphism (as then Theorem \ref{propertiesPAC} yields that $(L,\D)\equiv_{E}(F,\D)$). To see this, one simply needs to note that, since the polynomial
\begin{equation}\label{irredpoly}
x^{N(n)} + c_{n,N(n)-1}x^{N(n)-1}+\cdots+ c_{n,0}
\end{equation}
is irreducible over $L$, we have that $[E_n:E]=[E_nL:L]$ where $E_n$ is the Galois extension of $E$ (of degree $N(n)$) associated to the polynomial \eqref{irredpoly}. It follows that $E^{\alg}\cap L=E$ and $L^{\alg}=LE^{\alg}$. These equalities imply that the restriction $\gal(L)\to \gal(E)$ is an isomorphism.

\medskip

(3) Follows from (2).

\medskip

(4) By standard arguments (which are spelled out in Section \ref{imagine} below for the theory $\GDCF$). 

\medskip

(5) For $A, B, E$ subsets of $F$, say that $A$ and $B$ are independent over $E$ if $\acl(AE)$
and $\acl(BE)$ are linearly disjoint over $\acl(E)$. Note that, by (2), $\acl(E)$ is the relative field-theoretic algebraic closure in $F$ of the differential field generated by $E$ and $c$. Since this notion of independence in $T_{\D,c}$ is derived from independence in $\DCFm$, it satisfies the characteristic properties of (non-)forking in simple theories and, in particular, coincides with Shelah's notion of (non-)forking. It only remains to justify the claim on the $\SU$-rank of $T_{\D,c}$ being $\omega^m$. This is similar to the proof of Theorem~\ref{SUrankinv} above; that is, we only need to show that a saturated elementary extension of $(F,\D,c)$ contains a differentially transcendental element. To see this, let $f(x)$ be a differential polynomial over $F$. Since $(F,\D)$ is differentially large (Remark \ref{PACdiff}), Proposition 5.9 of \cite{OmarMarcus} yields that there is $a\in F$ such that $f(a)\neq 0$. The claim now follows by compactness.

\medskip

(6) By standard arguments, using the fact that $\DCFm$ eliminates imaginaries (these arguments are also spelled out in Section \ref{imagine} below for the theory $\GDCF$). 
\end{proof}

As a consequence, we get the following elimination of imaginaries result for the subfield of $G$-invariants of a model of $\GDCF$ (recall that we have seen that these subfields are bounded and PAC-differential). 

\begin{corollary}\label{cor:EI.PAC.diff}
Let $(K,\D,\S)\models \GDCF$ and $K^G$ be the subfield of invariants. Let $c$ be the (infinite) tuple from $K^G$ defined above (using boundedness of $K^G$). Then, the theory $\theo(K^G, \D,c)$ has elimination of imaginaries.
\end{corollary}

\

\subsection{Elimination of imaginaries for $\GDCF$}\label{imagine} Let $(F,\Delta,\Sigma)$ be a monster 
model of $\GDCF$.
By Corollary \ref{fromTCF}, $F$ is a bounded PAC-field of characteristic $0$, so we can choose an infinite tuple $c$ from $F$ coding boundedness of $F$ as was done in Section \ref{boundPACdiff} above and in 4.6 in \cite{ChPi}.
We consider a new structure
$$(F,\Delta,\Sigma,c).$$
In this section we show that the theory $T_{\D,\S,c}:=\theo(F,\Delta,\Sigma,c)$ has (full) elimination of imaginaries.
We will work in the language $\mathcal L_{\D,\S,c}$ which is the language of $G$-differential fields $\mathcal{L}_{\Delta,\Sigma}$ expanded by the symbols corresponding to the tuple of parameters $c$. For simplicity, we use ``$\ind$" instead of ``$\ind^{\GDCF}$". Moreover, as in previous sections, we let $(\U,\Delta)$ be a monster model of $\DCF_{0,m}$, containing $(F,\D)$, which we assume to be $|F|^+$-strongly homogeneous and $|F|^+$-saturated.

\begin{lemma}\label{lemma:good.acl}
Assume that $A\subseteq B\subseteq F$ are $G$-differential subfields of $F$ and $\acl^{T_{\D,\S,c}}(A)=A$.
Then
$$B^{\alg}=\acl^{T_{\D,\S,c}}(B)\cdot A^{\alg}.$$
In other words, $B^{\alg}$ is equal to the compositum of the algebraic closure of $B$ inside the $\mathcal{L}_{\D,\S,c}$-structure $F$ and the ``full" algebraic closure of $A$.
\end{lemma}

\begin{proof}
The proof of the second point of Proposition 4.6 in \cite{ChPi} shows that the restriction maps $\gal(F)\to \gal(A)$ and $\gal(F)\to \gal(\acl^{T_{\D,\S,c}}(B))$ are isomorphisms. Therefore the restriction map $G(\acl^{T_{\D,\S,c}}(B))\to G(A)$ is also an isomorphism. The result now follows. 
\end{proof}

\begin{lemma}\label{lemma:ind.alg}
In the theory $T_{\D,\S,c}$, the Independence Theorem over algebraically closed substructures holds. Namely, let $E=\acl^{T_{\D,\S,c}}(E)\subseteq F$, 
$c_1\equiv_E c_2$,
$c_1\ind_E a$, $c_2\ind_E b$ and $a\ind_E b$.
Then there exists $c'\in F$ such that
$c'\equiv_{Ea} c_1$, $c'\equiv_{Eb} c_2$ and $c'\ind_E ab$.
\end{lemma}

\begin{proof}
The proof is a modification of some standard proofs (e.g., the proof of Theorem 4.21 from \cite{Hoff3} or the proof of Theorem 4.7 from \cite{ChPi}), but we provide a sketch to clarify how the different assumptions interplay in our situation.
The main ingredients here are the boundary property $B(3)$ from $\DCF_{0,m}$ and Lemma \ref{lemma:good.acl}.

By the properties of $\ind$, we assume that we also have $c_1\ind_E ab$.
Let $A:=\acl^{T_{\D,\S,c}}(Ea)$, $B:=\acl^{T_{\D,\S,c}}(Eb)$, $C_1:=\acl^{T_{\D,\S,c}}(Ec_1)$, $C_2:=\acl^{T_{\D,\S,c}}(Ec_2)$
and $D:=\acl^{T_{\D,\S,c}}(Eab) \acl^{T_{\D,\S,c}}(Eac_1)$.

\bigskip
\noindent\textbf{Claim.} $D$ is regular over $B C_1$.
\begin{proof}[Proof of Claim] We need to show that
$$\acl^{T_{\D,\S,c}}(Eab) \acl^{T_{\D,\S,c}}(Eac_1)\;\cap\;(BC_1)^{\alg}=B C_1.$$
Because $\acl^{T_{\D,\S,c}}$ is equal to $\acl^{\GDCF}$ computed after adding parameters from $c$, we can rewrite the left-hand side of the above line as: 
$$\acl^{\GDCF}(AB) \acl^{\GDCF}(AC_1)\;\cap\;(BC_1)^{\alg},$$
which is a subset of $(AB)^{\alg}(AC_1)^{\alg}\cap(BC_1)^{\alg}$, which is equal - by the boundary property $B(3)$ - to $B^{\alg}  C_1^{\alg}$. 
Now, by Lemma \ref{lemma:good.acl}, we have that $B^{\alg}=B  E^{\alg}$ and $C_1^{\alg}=C_1  E^{\alg}$, so finally the aforementioned left-hand side is contained in $B  C_1  E^{\alg}$.
If $m\in \acl^{T_{\D,\S,c}}(Eab) \acl^{T_{\D,\S,c}}(Eac_1)\;\cap\;(BC_1)^{\alg}$, then $m\in F\cap(BC_1)^{\alg}=\acl^{T_{\D,\S,c}}(BC_1)$ and $m\in B  C_1  E^{\alg}$. 
We know that $E$ is relatively algebraically closed in $\acl^{T_{\D,\S,c}}(BC_1)$, so $E\subseteq \acl^{T_{\D,\S,c}}(BC_1)$ is regular. If $f\in\aut_{\Delta}(\U/BC_1)$, then - by Fact 3.33 from \cite{Hoff3} - there exists $h\in\aut_{\Delta}(\U/BC_1E^{\alg})$ such that 
$$h|_{\acl^{T_{\D,\S,c}}(BC_1)}=f|_{\acl^{T_{\D,\S,c}}(BC_1)},$$ thus $f(m)=h(m)=m$. Therefore $m\in\dcl^{\DCF_{0,m}}(BC_1)=B  C_1$ and that finishes the proof of the claim.
\end{proof}

Now, stability of $\DCF_{0,m}$ comes into the picture.
Because $c_1\equiv_E c_2$, there is $f_0\in\aut_{\D,\S,c}(F/E)$ such that $f_0(c_1)=c_2$.
Then, because $C_1\ind_E B$, $C_2\ind_E B$ and $B$ is regular over $E$, Corollary 3.38 from \cite{Hoff3} implies the existence of $f\in\aut_{\Delta}(\U/B)$ such that $f|_{C_1}=f_0|_{C_1}$, hence $f:\langle C_1,B\rangle\to \langle C_2,B\rangle$ is an isomorphism  of $\mathcal L_{\D,\S,c}$-structures (indeed, it preserves the $G$-differential structure and parameter $c$).
Now, let $C:=f^{-1}(\acl^{T_{\D,\S,c}}(C_2B))$, we have that $\langle C_1,B\rangle\subseteq C\subseteq \acl^{T_{\D,\S,c}}(C_1B)$ and that
$f:(C,\Sigma')\to (\acl^{T_{\D,\S,c}}(Ec_2b),\Sigma)$, where $\Sigma'=(f^{-1}\sigma f)_{\sigma\in\Sigma}$,
is an isomorphism of $\mathcal L_{\D.\S,c}$-structures (note that $\Sigma'|_{\langle C_1,B\rangle}=\Sigma|_{\langle C_1,B\rangle}$).

Let, for each $\sigma\in\Sigma$, $h_{\sigma}\in\aut_{{\Delta}}(\U)$ extend $f^{-1}\sigma f|_{C}$.
Because $BC_1\subseteq D$ is regular, Fact 3.33 from \cite{Hoff3}, allows us to find $\hat{\sigma}\in\aut_{\Delta}(\U)$ such that $\hat{\sigma}|_D=\sigma|_D$ and $\hat{\sigma}|_{(BC_1)^{\alg}}=h_{\sigma}|_{(BC_1)^{\alg}}$. We set $\hat{\Sigma}=(\hat{\sigma})_{\sigma\in\Sigma}$
and note that $(\langle D,C\rangle,\Delta,\hat{\Sigma})$ extends, simultaneously, the $G$-differential structures of $(D,\Delta,\Sigma)$ and $(C,\Delta,\Sigma')$. We embed $(\langle D,C\rangle,\Delta,\hat{\Sigma})$ into some model of $\GDCF$,
say $(F'',\Delta'',\Sigma'')$. Quantifier elimination of $\DCF_{0,m}$ allows us to view $(F'',\Delta'')$ inside $(\U,\Delta)$. Because $(B,\Delta,\Sigma)$ is a common substructure of $(F,\Delta,\Sigma)$ and of $(F'',\Delta'',\Sigma'')$, Remark \ref{remark:completeness}, implies that $(F,\Delta,\Sigma)$ and $(F'',\Delta'',\Sigma'')$ have the same $\mathcal{L}_{\Delta,\Sigma}$-theory over $B$, 
and since $c$ is contained in $B$, also the same $\mathcal L_{\D,\S,c}$-theory over $B$. Therefore we may embed $(F'',\Delta'',\Sigma'')$ into $(F,\Delta,\Sigma)$ over $B$ as an elementary $\mathcal L_{\D,\S,c}$-substructure.
Let $a',c_1'$ denote the images (under the described embedding) of $a,c_1$ respectively.
By Fact 4.9 from \cite{Hoff3}, we have that $a'b\equiv_E ab$, $a'c_1'\equiv_E ac_1$ and $bc_1'\equiv_E bc_1$ (in $\mathcal L_{\D,\S,c}$). Hence, there exists $g\in\aut_{\D,\S,c}(F/Eb)$ such that $g(a')=a$. For $c':=g(a')$, we have $c'\equiv_{Ea} c_1$ and
$c'\equiv_{Eb} c_2$. It is left to show that $c'\ind_E ab$.

We know that $c_1\ind_E ab$. We placed both $(F,\Delta)$ and $(F'',\Delta'')$ inside $(\U,\Delta)$
and we know that the forking independence in $\GDCF$ is given by the forking independence in $\DCF_{0,m}$ (after taking orbits of the $G$-action). Therefore $c_1\ind_E ab$ transforms into $c_1'\ind_E a'b$, which becomes (after acting via $g$) $c'\ind_E ab$.
\end{proof}

\begin{theorem}\label{thm:full.EI}
The $\mathcal L_{\D,\S,c}$-theory $T_{\D,\S,c}$ has elimination of imaginaries.
\end{theorem}

\begin{proof}
Recall that $(F,\Delta,\Sigma,c)$ is a monster model of $T_{\D,\S,c}$.
Let $e$ be an element of $(F,\Delta,\Sigma,c)^{\eq}$ given by a $\emptyset$-definable function $f$ and a finite tuple $a\subseteq F$, i.e. $f(a)=e$. Let $C:=\acl^{T_{\D,\S,c}^{\eq}}(e)\cap F$
and let $Q:=\aut_{\D,\S,c}(F/C)\cdot a$. 
The claim from the proof of Theorem 4.36 from \cite{Hoff3} holds in our situation, hence there is $u\in Q$ such that $f(u)=e$ and $a\ind_C u$.

Let $D:=\{d\in Q\;|\;f(d)=e\}$.
If $D=Q$, then $e\in\dcl^{\GDCF^{\eq}}(C)$ and we can find an element $b\in C$ such that $e\in\dcl^{\GDCF^{\eq}}(b)$.
Since $b\in C=\acl^{\GDCF^{\eq}}(e)$, we obtain weak elimination of imaginaries, which combined with the fact that we are in theory of fields (hence finite sets can be coded),
gives us full elimination of imaginaries in $T_{\D,\S,c}$.

Now, we assume that $D\subsetneq Q$.
Take some $d_0\in Q\setminus D$ and $d\equiv_C d_0$ such that $d\ind_C D$.
If $f(d)=e$, then $d\in D$ and we obtain that $d\in C$, hence again $e\in \dcl^{\GDCF^{\eq}}(C)$ and we get full elimination of imaginaries in $T$. Suppose that $f(d)\neq e$. We have that
$a,u,d$ have the same type over $C$ in $T$, 
$d\ind_C au$, $a\ind_C u$, and $f(a)=f(u)=e$. 

In the next step, we would like to use the Independence Theorem, but over the set $C$ instead over a model - this is the reason for providing the above Lemma \ref{lemma:ind.alg}.
We have $\tp^{T_{\D,\S,c}}(u/C)=\tp^{T_{\D,\S,c}}(d/C)$, non-forking extensions $\tp^{T_{\D,\S,c}}(u/C)\subseteq\tp^{T_{\D,\S,c}}(u/Ca)$
and $\tp^{T_{\D,\S,c}}(d/C)\subseteq\tp^{T_{\D,\S,c}}(d/Cu)$, and we know that $a\ind_C u$. Therefore (by Lemma \ref{lemma:ind.alg}) there exists $m\in F$ such that $m\equiv_{Ca}u$, $m\equiv_{Cu}d$ and $m\ind_C au$.
But then $f(m)=f(a)=f(u)\neq f(m)$, so it can not happen that $f(d)\neq e$.
\end{proof}

\begin{remark}
In Theorem 4.10 from \cite{HoffKow2018} it is not clear whether it is enough to add only parameters indicated in the statement.
Now, our Theorem \ref{thm:full.EI} implies that the theory $G$-TCF (working in characteristic zero case) has (full) elimination of imaginaries after adding the tuple $c$ as above - it is enough to put $m=0$ (the number of derivations) in $\GDCF$.
\end{remark}

\section{Final remarks}\label{sec:final}
In this last section, we briefly touch on the general case of the theory of $G$-differential fields in characteristic zero 
for an \emph{arbitrary} group $G$ (not necessarily finite). As for the case of finite $G$, we denote this theory by ``$\GDF$" and its model companion, \emph{if it exists}, by ``$\GDCF$".
A natural question is for which $G$ the model companion $\GDCF$ exists. It is not an easy question even in the case of $m=0$ (``pure" fields with $G$-action) and only partial results have been achieved (see \cite{Sjo} and \cite{BeyKow2018}, \cite{BK2}). We start with a negative observation:

\subsection{Two commuting automorphisms}
By an argument of Hrushovski, it is well known that the theory $\GTCF$ does not exist for $G=\mathbb{Z}\times\mathbb{Z}$
(the model companion of the theory of fields with two commuting automorphisms). The proof was published, for example, in \cite{Kikyo1} and we may adapt it to show that similarly the theory $\GDCF$ does not exist for $G=\mathbb{Z}\times\mathbb{Z}$. Instead of rewriting the whole proof, we will only indicate the key differences.
Here, we set $T$ to be theory of fields with $m$ commuting derivations and two commuting automorphism (also commuting with the derivations; that is, they are differential automorphisms).

First, we need to adapt Lemma 3.1 from \cite{Kikyo1} to our difference-differential setting.
Our adaptation says that for any existentially closed model $(F,\Delta,\sigma,\tau)$ of $T$, for any $n\geqslant 2$
there is a $c\in \ker \D$ such that $\sigma(c)=\tau(c)$, $c+\sigma(c)+\ldots+\sigma^{n-1}(c)=0$ and
$c+\sigma(c)+\ldots+\sigma^{k-1}(c)\neq 0$ for any $k<n$. The argument from \cite{Kikyo1} works in our case if we 
extend the differential structure to $F(t_0,\ldots,t_{n-2})$ via $\delta_i(t_j)=0$, where $i\leqslant m$, $j\leqslant n-2$.

Similarly, we need to adapt the proof of Theorem 3.2 from \cite{Kikyo1} to obtain the following:

\begin{theorem}
For $G=\mathbb{Z}\times\mathbb{Z}$, the theory $\GDF$ does not have a model-companion.
\end{theorem}

\begin{proof}
We follow the lines from the proof of Theorem 3.2 in \cite{Kikyo1}. The first difference is that instead of $(K_0,\sigma_0,\sigma_0)$ we consider $(K_0,\Delta_0,\sigma_0,\sigma_0)$, where $\Delta_0$ are trivial derivations (i.e. $\ker\Delta_0=K_0$). We embed it into a somehow saturated model $(K,\Delta,\sigma,\tau)$ of the theory of $\GDCF$.

The second difference shows up in the Claim 3.2.1 there, which takes the following form:
$$\sigma(z)=\tau(z),\,z+\sigma(z)+\ldots+\sigma^k(z)\neq 0\text{ for }k<\omega,\;z\in\ker\Delta$$
$$\vdash(\exists x) (\exists y)\big(\sigma(x)=\tau(x)=x+z\;\wedge\;y^3=x\;\wedge\;\tau(y)=\zeta\sigma(y) \big)$$
The proof of the new version of the claim requires extending the differential structure onto $E(a)$, which we do by
$\delta'_i(a)=0$ for every $i\leqslant m$ (then $\sigma'$, $\tau'$ and all $\delta'_i$'s commute on $E(a)$).
After this, we need to define the value of each $\delta'_j$ on every $b_i$, but there is no choice and it needs to be $\delta'_j(b_i)=0$. This way we obtain the claim.

The rest of the proof from \cite{Kikyo1} goes through if we use our versions of the claim and Lemma 3.1 (described above).
\end{proof}

\subsection{On the bright side}
We now know that $\GDCF$ exists for finite groups; and also for $G$ being a finitely generated free group \cite{OmarRahim}.
So the situation is similar to that of the theory $\GTCF$ before  the manuscript \cite{BeyKow2018}, where the authors combined finite groups and finitely generated free groups into virtually free groups and showed existence of $\GTCF$ for every $G$ finitely generated virtually free group. Our intuition is that we one might be able to adapt methods from \cite{BeyKow2018} to obtain that $\GDCF$ exists for all finitely generated virtually free groups. However, at the moment, it is unclear how to generalize Cohn's difference-differential basis theorem \cite{Cohn1970} to the setup with difference structure given by the action of a group $G$. 
We intend to investigate this problem in a forthcoming manuscript.

\medskip

We wish to note that the existence of $\GDCF$ for infinite groups serves for the purpose of developing neostability; especially in the class of simple and NSOP$_1$ theories. The crucial observation here is that the model-theoretic properties of models of $\GDCF$ (if it exists) depend on the properties of their absolute Galois groups (as in \cite{Hoff3}), and the properties of their absolute Galois groups depend on the properties of the group $G$ (which is the case for theories with group actions, see Section 5 in \cite{Hoff4}). For example,
we have the following interesting result, expressing that simplicity (on different levels) is equivalent to having a small absolute Galois group (i.e. boundedness):

\begin{proposition}\label{prop:equivalent}
Assume that $G$ is a group (possibly infinite) and that $\GDCF$ exists (i.e., the theory of $G$-differential fields in characteristic zero has a model-companion). Let $(K,\Delta,\Sigma)\models\GDCF$.
The following are equivalent:
\begin{enumerate}
    \item $K$ is bounded,
    \item $K$ is simple,
    \item $(K,\Delta)$ is simple,
    \item $(K,\Delta,\Sigma)$ is simple.
\end{enumerate}
\end{proposition}

\begin{proof}
In a natural way, we have that (4)$\Rightarrow$(3)$\Rightarrow$(2).
We know that $K$ is a PAC-field: the case of finite $G$ is covered by Corollary \ref{fromTCF}, the general case follows from
Proposition 3.56 from \cite{Hoff3} (which yields that $(K,\D)$ is PAC-differential) combined with Remark \ref{PACdiff} (to get that $K$ is a PAC-field).
Thus $K$ is simple if and only if $K$ is bounded, so (1)$\iff$(2). We are done if we show (1)$\Rightarrow$(4).

Assume that $K$ is bounded. Because the absolute Galois group of $K$ in ACF is equal to the absolute Galois group of $K$ in the sense of $\DCF_{0,m}$, we have that $(K,\Delta,\Sigma)$ is bounded or equivalently \emph{Galois bounded} (see Definition 3.48 in \cite{Hoff3}). By Corollary 4.28 from \cite{Hoff3}, it follows that $(K,\Delta,\Sigma)$ is simple.
\end{proof}

Another fact in this direction can be obtained after a slight modification of Theorem 4.7 from \cite{BeyKow2018}:

\begin{proposition}
Assume that $G$ is a finitely generated virtually free group and assume that $\GDCF$ exists. Then:
\begin{enumerate}
    \item $\GDCF$ is simple if and only if $G$ is finite or $G$ is free,
    \item if $\GDCF$ is not simple, then it is not NTP$_2$.
\end{enumerate}
\end{proposition}

\begin{proof}
If $G$ is finite or free, we know that $\GDCF$ is simple (by the results in this paper and by \cite{OmarRahim}, respectively).
Assume that $(K,\Delta,\Sigma)\models\GDCF$.
If $G$ is infinite, but not free, then the kernel of the Frattini cover of the profinite completion of $G$ is not a small profinite group (Theorem 4.6 in \cite{BeyKow2018}). On the other hand, the absolute Galois group (in $\DCFm$) of $K\cap (K^G)^{\alg}$ is isomorphic to this kernel (by Corollary 5.6 in \cite{Hoff4}), so $\aut_{\Delta}((K^G)^{\alg}/K\cap (K^G)^{\alg})$ is not a small profinite group. The restriction map 
$$\aut_{\Delta}(K^{\alg}/K)\to\aut_{\Delta}((K^G)^{\alg}/K\cap (K^G)^{\alg})$$
is onto (by Lemma 2.15 from \cite{Hoff4}), thus also $\aut_{\Delta}(K^{\alg}/K)=\aut(K^{\alg}/K)$ is not small. Then, Proposition \ref{prop:equivalent} implies that $(K,\Delta,\Sigma)$ is not simple.

Now we are going to prove the second item. If $(K,\D,\S)\models\GDCF$ is not simple, then $K$ is not bounded (by Proposition \ref{prop:equivalent}). However, as we saw in the proof of Proposition \ref{prop:equivalent} we know that $K$ is a PAC-field, and hence, by a result of Chatzidakis (Section 3 in \cite{Chatzi}), the theory of $K$ (as a field) is not NTP$_2$. Therefore, $(K,\D,\S)$ cannot be NTP$_2$.

\end{proof}

Due to the above result, existence of $\GDCF$ for infinite not-free finitely generated virtually free groups (e.g., the infinite dihedral group) may be relevant for the model theory of non-simple NSOP$_1$ structures, provided we can show a suitable analogue of Proposition \ref{prop:equivalent} (so we would like to find a relation between properties of the absolute Galois group of a model of $\GDCF$ and being an NSOP$_1$ structure). Such an analogue already exists in the case of PAC-fields (see Corollary 7.2.7 in \cite{NickThesis}), and more generally in the case of PAC structures (see Theorem 6.11 in \cite{HoffLee}), so one could expect that there will be some variant of such a result for PAC structures (or at least for PAC-differential fields) equipped with a group action (compare with the last paragraphs in \cite{HoffLee}).

\bigskip

\bibliographystyle{plain}
\bibliography{harvard1}

\begin{thebibliography}{10}

\bibitem{CCMA}
Martin Bays, Martin Hils, and Rahim Moosa.
\newblock Model theory of compact complex manifolds with an automorphism.
\newblock {\em Trans. Amer. Math. Soc.}, 369(6):4485--4516, 2017.

\bibitem{BK2}
\"{O}zlem Beyarslan and Piotr Kowalski.
\newblock Model theory of galois actions of torsion abelian groups.
\newblock Submitted, available on \url{https://arxiv.org/pdf/2003.02329.pdf}.

\bibitem{BeyKow2018}
\"{O}zlem Beyarslan and Piotr Kowalski.
\newblock Model theory of fields with virtually free group actions.
\newblock {\em Proc. London Math. Soc.}, 118(2):221--256, 2019.

\bibitem{Busta}
Ronald Bustamante.
\newblock Differentially closed fields of characteristic zero with a generic
  automorphism.
\newblock {\em Revista de Matem\'atica: Teor\'ia y Aplicaciones},
  81(2):81--100, 2007.

\bibitem{ChPi}
Z.~Chatzidakis and A.~Pillay.
\newblock Generic structures and simple theories.
\newblock {\em Annals of {P}ure and Applied Logic}, 95:71--92, 1998.

\bibitem{Chatzi}
Zoe Chatzidakis.
\newblock Independence in (unbounded) pac fields, and imaginaries.
\newblock {\em Preprint, available on
  http://www.logique.jussieu.fr/~zoe/papiers/Leeds08.pdf}, 2008.

\bibitem{acfa1}
Zo{\'e} Chatzidakis and Ehud Hrushovski.
\newblock Model theory of difference fields.
\newblock {\em Trans. AMS}, 351(8):2997--3071, 2000.

\bibitem{ChaHr1}
Zoe Chatzidakis and Ehud Hrushovski.
\newblock Difference fields and descent in algebraic dynamics. {I}.
\newblock {\em Journal of the Institute of Mathematics of Jussieu},
  7(4):653--686, 2008.

\bibitem{Cohn1970}
Richard Cohn.
\newblock A difference-differential basis theorem.
\newblock {\em Canadian Journal of Mathematics}, 22(6):1224--1237, 1970.

\bibitem{DHL}
J.~Dobrowolski, D.M. Hoffmann, and Lee J.
\newblock Elementarily {E}quivalence {T}heorem for {P}{A}{C} structures.
\newblock {\em Journal of Symbolic Logic}, 2020.
\newblock https://arxiv.org/pdf/1811.08510.pdf.

\bibitem{Hoff3}
D.M. Hoffmann.
\newblock Model theoretic dynamics in a {G}alois fashion.
\newblock {\em Annals of Pure and Applied Logic}, 170(7):755--804, 2019.

\bibitem{Hoff4}
D.M. Hoffmann.
\newblock On {G}alois groups and {P}{A}{C} substructures.
\newblock {\em Fundamenta Mathematicae}, 250:151--177, 2020.

\bibitem{HoffLee}
D.M. Hoffmann and Lee J.
\newblock Co-theory of sorted profinite groups for {P}{A}{C} structures.
\newblock Available on
  \url{https://drive.google.com/file/d/15hyjAzpzwM_UmgnHL_AzBDhu9Ft95YcH/view}.

\bibitem{HoffKow2018}
D.M. Hoffmann and P.~Kowalski.
\newblock Existentially closed fields with finite group actions.
\newblock {\em Journal of Mathematical Logic}, 18(1):1850003, 2018.

\bibitem{manuscript}
Ehud Hrushovski.
\newblock {\em Pseudo-finite fields and related structures}, volume~11 of {\em
  Quad. Mat.}
\newblock Aracne, Rome, 2002.

\bibitem{Kikyo1}
Hirotaka Kikyo.
\newblock On generic predicates and automorphisms.
\newblock Available on
  \url{http://repository.kulib.kyoto-u.ac.jp/dspace/bitstream/2433/25825/1/1390-1.pdf}.

\bibitem{KiPi}
Hirotaka Kikyo and Anand Pillay.
\newblock The definable multiplicity property and generic automorphisms.
\newblock {\em Annals of Pure and Applied Logic}, 106:263--273, 2000.

\bibitem{Kimsim}
B.~Kim.
\newblock {\em Simplicity Theory}.
\newblock Oxford Logic Guides. OUP Oxford, 2013.

\bibitem{kol1}
E.R. Kolchin.
\newblock {\em Differential Algebra and Algebraic Groups}.
\newblock Pure and applied mathematics. Academic Press, 1973.

\bibitem{Kolchin2}
E.R. Kolchin.
\newblock {\em Differential Algebraic Groups}.
\newblock Academic Press, 1985.

\bibitem{McGrail}
Tracey McGrail.
\newblock The model theory of differential fields with finitely many commuting
  derivations.
\newblock {\em The Journal of Symbolic Logic}, 65(02):885--913, 2000.

\bibitem{med1}
Alice Medvedev.
\newblock $\mathbb{Q}${ACFA}.
\newblock Available on \url{http://arxiv.org/abs/1508.06007}.

\bibitem{MoosaScanlon}
Rahim Moosa and Thomas Scanlon.
\newblock Model theory of fields with free operators in characteristic zero.
\newblock {\em Journal of Mathematical Logic}, 14(2):5991--6016, 2014.

\bibitem{PiPolk}
Anand Pillay and Dominika Polkowska.
\newblock On {P}{A}{C} and bounded substructures of a stable structure.
\newblock {\em The Journal of Symbolic Logic}, 71(2):460--472, 2006.

\bibitem{Pop2014}
Florian Pop.
\newblock Little survey on large fielss - old and new.
\newblock {\em EMS Series of Congress Reports: Valuation theory in
  interaction}, Editors: Campillo, Kuhlmann, Teissier, 2014.

\bibitem{NickThesis}
Samuel~Nicholas Ramsey.
\newblock {\em Independence, {A}malgamation, and {T}rees}.
\newblock PhD thesis, UC Berkeley, Group in Logic and the Methodology of
  Science, 2018.

\bibitem{OmarPhD}
O.~Le\'on S\'anchez.
\newblock {\em Contributions to the model theory of partial differential
  fields}.
\newblock PhD thesis, University of Waterloo, 2013.

\bibitem{OmarDCFA}
O.~Le\'on S\'anchez.
\newblock On the model companion of partial differential fields with an
  automorphism.
\newblock {\em Israel Journal of Mathematics}, 212(1):419--442, 2016.

\bibitem{OmarJames}
O.~Le\'on S\'anchez and J.~Freitag.
\newblock Effective uniform bounding in partial differential fields.
\newblock {\em Advances in Mathematics}, 288:308--336, 2016.

\bibitem{OmarMarcus}
O.~Le\'on S\'anchez and Tressl M.
\newblock Differentially large fields.
\newblock Available on \url{https://arxiv.org/abs/2005.00888}.

\bibitem{OmarRahim}
O.~Le\'on S\'anchez and Rahim Moosa.
\newblock The model companion of differential fields with free operators.
\newblock {\em The Journal of Symbolic Logic}, 81(2):493--509, 2016.

\bibitem{Sjo}
Nils Sj\"{o}gren.
\newblock The {M}odel {T}heory of {F}ields with a {G}roup {A}ction.
\newblock {\em Research Reports in Mathematics, Department of Mathematics
  Stockholm University}, 2005.
\newblock Available on \url{http://www2.math.su.se/reports/2005/7/}.

\bibitem{Tre2005}
Marcus Tressl.
\newblock The uniform companion for large differential fields of characteristic
  $0$.
\newblock {\em Transactions of the American Mathematical Society},
  357(10):3933--3951, 2005.

\bibitem{vddsch}
L.~van~den Dries and K.~Schmidt.
\newblock Bounds in the theory of polynomial rings over fields. {A} nonstandard
  approach.
\newblock {\em Inventiones mathematicae}, 76(1):77--91, 1984.

\end{thebibliography}

\end{document}